\documentclass{article}

\usepackage[utf8]{inputenc}

\usepackage{amsmath}

\usepackage{amssymb}
\usepackage{amsthm}
\usepackage{geometry}
\usepackage[hidelinks]{hyperref}
\usepackage{graphicx}
\usepackage{xcolor}
\usepackage{float}

\usepackage[style = numeric, backend = biber, minnames = 5, maxnames= 99]{biblatex}
\newtheorem{proposition}{Proposition}[section]
\newtheorem{theorem}[proposition]{Theorem}
\newtheorem{lemma}[proposition]{Lemma}
\newtheorem{corollary}[proposition]{Corollary}

\theoremstyle{definition}
\newtheorem*{definition}{Definition}

\theoremstyle{remark}
\newtheorem*{remark}{Remark}
\numberwithin{equation}{section}

\newcommand{\R}{\mathbb{R}}

\newcommand{\Z}{\mathbb{Z}}

\newcommand{\N}{\mathbb{N}}

\newcommand{\Q}{\mathbb{Q}}

\newcommand{\C}{\mathbb{C}}

\renewcommand{\Im}[1]{\,\mathrm{Im}\left(#1\right)}

\renewcommand{\Re}[1]{\,\mathrm{Re}\left(#1\right)}

\newcommand{\pdiff}[2]{\frac{\partial #1}{\partial #2}} 

\newcommand{\hipdiff}[3]{\frac{\partial^#1 #2}{\partial #3 ^{#1}}}
\makeatletter
\def\moverlay{\mathpalette\mov@rlay}
\def\mov@rlay#1#2{\leavevmode\vtop{%
		\baselineskip\z@skip \lineskiplimit-\maxdimen
		\ialign{\hfil$\m@th#1##$\hfil\cr#2\crcr}}}
\newcommand{\charfusion}[3][\mathord]{
	#1{\ifx#1\mathop\vphantom{#2}\fi
		\mathpalette\mov@rlay{#2\cr#3}
	}
	\ifx#1\mathop\expandafter\displaylimits\fi}
\makeatother

\newcommand{\modc}[1]{\,\left(\mathrm{mod}\, #1\right)}

\newcommand{\mc}[1]{\mathcal{#1}}
\newcommand{\mf}[1]{\mathfrak{#1}}
\newcommand{\mr}[1]{\mathrm{#1}}
\newcommand{\mbb}[1]{\mathbb{#1}}
\newcommand{\bs}{\backslash}
\newcommand{\labelunder}[2]{\underset{#1}{\underbrace{#2}}}
\newcommand\smalltwosqmat[4]{\left(\begin{smallmatrix} #1 & #2\\ #3 & #4 \end{smallmatrix} \right)}

\author{Kush Singhal}
\title{The Completed $L$-function attached to the Weight 2 Polar Harmonic Maass Form $H_{N,z}^*(\tau)$}
\date{}

\begin{document}
	\maketitle
	\begin{abstract}
		In this paper, we study the Mellin transform of the weight 2 level $N$ polar harmonic Maass form $H_{N,z}^*(\tau)$, and analyze this (generalized) $L$-function as $\Im{z}\to \infty$. On the way, we also calculate the Fourier expansion of $H_{N,z}^*(\tau)$ at arbitrary cusps of $\Gamma_0(N)$, and we give a functional equation and factorization into local factors of the $L$-function for the weight 2 level $N$ Eisenstein series at the cusps $i\infty$ and $0$. 
	\end{abstract}
	
	\tableofcontents
	
	\section{Introduction and Statement of Results}\label{sctn::intro}
	We consider the polar harmonic Maass form $H_{N,z}^*(\tau)$, where $N\in \N$ and $z, \tau\in \mbb H$. These forms were first constructed by Bringmann and Kane \cite{construction_Maass_form} as the analytic continuation to $s=0$ of a certain Poincar\'e series (see Subsection \ref{sctn::prelim_HNz} for details). 
	As a function of $\tau$, the only poles of $H_{N,z}^*(\tau)$ are simple poles at $\tau = \gamma z$ for any $\gamma \in \Gamma_0(N)$ with residue $$\frac{1}{4\pi i} \cdot \#\mr{Stab}_z(\Gamma_0(N)),$$
	where $\Gamma_0(N)$ denotes the usual congruence subgroup of $SL_2(\Z)$. This useful fact allows one to study the divisors of meromorphic modular forms; explicitly, it was shown by Bringmann, Kane, L\"obrich, Ono, and Rolen \cite{divisors_mod_forms} that for any weight $k$ meromorphic modular form $f$ on $\Gamma_0(N)$, the following \emph{divisor polar harmonic Maass form} can be written as \[\sum_{z\in \Gamma_0(N) \bs \mbb H} e_{N,z} \mr{ord}_z(f) H_{N,z}^*(\tau) =: f^{\mr{div}}(\tau) = \frac{k}{4\pi \Im{\tau}} - \frac{1}{2\pi i} \frac{f'(\tau)}{f(\tau)} + g(\tau)\]
	for some weight 2 cusp form $g(\tau)$ on $\Gamma_0(N)$, where $f'(\tau)$ is the usual derivative of $f$ and $e_{N,z} := 2/\#\mr{Stab}_z(\Gamma_0(N))$ (with $e_{N,\rho} :=1$ for cusps $\rho$). Due to the slow-growing nature of the coefficients of the Fourier expansion of cusp forms, it is possible to compute the divisors of meromorphic modular forms numerically. The interested reader may refer to \cite{divisors_mod_forms} for details and examples of this computation. 
	
	For an explicit expression of $H_{N,z}^*(\tau)$, we may look to its Fourier expansion at different cusps. However, the presence of the simple pole at $z$ (and at all points $\Gamma_0(N)$-equivalent to $z$) complicates issues of convergence. Due to these poles, the Fourier expansion at a cusp $\rho$ only converges if $\tau$ is ``sufficiently close'' to the cusp, in the sense that $\Im{L \tau}$ needs to be sufficiently large where $L\in SL_2(\Z)$ is chosen so that $L(i \infty) = \rho$. 
	
	The Fourier expansion of $H_{N,z}^*$ at $i\infty$ was computed in \cite{divisors_mod_forms} to be \[H^*_{N,z}(\tau) = \frac{3}{\pi [SL_2(\Z) : \Gamma_0(N)] \Im \tau} + \sum_{n=1}^\infty j_{N,n}(z) e^{2\pi i n \tau},\]
	which converges only for $\Im{\tau} > \max\{\Im z, 1/\Im z\}$. Here the $j_{N,n}(z)$ are themselves weight 0 polar harmonic Maass forms for $\Gamma_0(N)$ whose only poles occur at the cusps. An explicit expression for $j_{N,n}(z)$ as a Fourier expansion in $z$ with coefficients involving Kloosterman sums and the $I$- and $J$- Bessel functions was also given in \cite{divisors_mod_forms}, and is reproduced as Proposition \ref{prop::fourier_expansion_infty} in Subsection \ref{sctn::prelim_HNz}. The Fourier expansion at any cusp for the map $z\mapsto H_{N,z}^*(\tau)$ (where $\tau$ is fixed) was also computed in \cite{construction_Maass_form}; this map is a polar harmonic Maass form of weight 0 for $\Gamma_0(N)$. In this paper, we compute the Fourier expansion of $H_{N,z}^*(\tau)$ at an arbitrary cusp $\rho$ of $\Gamma_0(N)$ (see Proposition \ref{prop::fourier_at_arb_cusps}). The computation proceeds along the same lines as in \cite{construction_Maass_form}. In fact, we directly use many of the computations as well as results on convergence from \cite{construction_Maass_form}. 
	
	The $j$-functions appearing as coefficients in the Fourier expansion at $i\infty$ are interesting in their own right. These are weight 0 polar harmonic Maass forms for $\Gamma_0(N)$. For $N=1$, these $j$-functions are closely related to the modular $j$-function \[j(\tau) = \frac{E_4(\tau)^3}{\Delta(\tau)} = e^{-2\pi i \tau} + 744 + 196884 e^{2\pi i \tau} + \cdots.\]
	In fact, $j_{1,1}(\tau) = j(\tau)$ and the functions $j_{1, n}(\tau)$ form a Hecke system, i.e. $j_{1,n}(\tau) = j_{1,1}(\tau) |T(n)$ where $T(n)$ is the $n$'the normalized Hecke operator. Similar relations involving the Hecke operator are true for arbitrary $N\in \N$ as well, though for $N>1$ some additional complications arise. The interested reader is referred to \cite{divisors_mod_forms} for the specific case of the $j_{N,n}(\tau)$ given here, and to \cite{j_heckesystems_general} for a more general discussion of Hecke systems of $j$-functions. 
	
	Another interesting property of the $H_{N,z}^*(\tau)$ polar harmonic Maass forms is their close connection to weight 2 Eisenstein series for $\Gamma_0(N)$. It was shown in \cite[][Theorem 1.2]{divisors_mod_forms} that for any cusp $\rho$ of $\Gamma_0(N)$, we have the following limit \[H_{N, \rho}^*(\tau) := \lim_{z\to \rho} H_{N,z}^*(\tau) = - E_{2,N,\rho}^*(\tau),\]
	where $E_{2,N,\rho}^*(\tau)$ is the weight 2 harmonic Eisenstein series which has a constant term 1 at $\rho$ and vanishes at all other cusps. At $N=1$ this is the usual weight 2 non-holomorphic Eisenstein series whose Fourier expansion is
	\[-H_{1,i\infty}^*(\tau) = E_2^*(\tau) = 1 -\frac{3}{\pi \Im{\tau}} - 24\sum_{m\ge 1} \sigma(m) e^{2\pi i m\tau},\]
	where $\sigma(m)$ is the sum of divisors of $m$. More generally, the Fourier coefficients $j_{N,n}(\rho)$ of $H_{N,\rho}^*$ are closely related to sums of divisor functions (see Corollary \ref{cor::explicit_formula_for_j_N,n_rho} below). As such, $j_{N,n}(\rho)$ grows like $O(n^{3/2})$, and so the Dirichlet series  
	\[\sum_{n\ge 1} \frac{j_{N,n}(i\rho)}{n^s}\]
	attached to $H_{N, i\rho }^*(\tau)$ is well-defined for $\Re s > \frac52$. For $N=1$ (and thus $\rho = i\infty$ and $j_{N,n}(i\infty) = 24 \sigma(n)$), this Dirichlet series has a meromorphic continuation to $\C$ given by \[24 \zeta(s) \zeta(s-1),\]
	where $\zeta(s)$ is the usual Riemann zeta function. This is in-fact a classical $L$-function. Recall that $L$-functions are meromorphic functions $L^\infty(s)$ that encode useful local ($p$-adic) information (in the form of a sequence $c(p^r)$) for a global object in the form of local factors $$L_p(s) = 1 + \sum_{r=1}^\infty \frac{c(p^r)}{p^{r s}}.$$
	Such a sum can often be rewritten in the form $1/(1 - f_p(p^{-s})p^{-s})$ for some polynomial $f_p$. These local factors then give the $L$-function via the Euler product \[L^\infty(s) = \prod_{p} L_p(s),\] where the product runs over all rational primes $p$, or more generally over all finite places if we are considering general number fields. Expending out the product and defining $c(n)$ multiplicatively then yields a Dirichlet series $$L^\infty(s) = \sum_{n\ge 1} \frac{c(n)}{n^s}.$$ 
	Usually, the Dirichlet series only converges on some half-plane, and so the $L$-function $L^\infty(s)$ gives a meromorphic continuation for the Dirichlet series.
		
	By considering the infinite place as well, we can also obtain a completed $L$-function $$L(s) = L_{\infty}(s) L^\infty(s) = L_{\infty}(s) \prod_p L_p(s),$$
	which usually has a meromorphic continuation to $\C$ and satisfies a certain functional equation of the form \[L(s) = \pm L'(k-s)\] 
	for some (completed) $L$-function $L'$ and some choice of sign. These functional equations are usually proved by showing that the completed $L$-function $L$ is the (suitably regularized) Mellin transform of either a modular form or a harmonic Maass form, and then using the functional equation given by modularity. For forms of level 1, they satisfy a functional equation involving the points $z$ and $-\frac{1}{z}$, which usually yields $L' = L$. However, for higher levels, $L'$ is usually different, and the pair $L$ and $L'$ usually encode information about the form at the cusps $0$ and $i\infty$. 
	
	For instance, for $H_{1, i\infty}^*(\tau)$, the function \[L(s) = \int_{t_0}^\infty t^{s-1}\left(H_{1,i\infty}^*(it) + 1 - \frac{3}{\pi t} \right)dt + \int_0^{t_0} t^{s-1} \left(H_{1,i\infty}^*(it) - \frac{1}{t^2} + \frac{3}{\pi t} \right)dt + \frac{t_0^s}{s} + \frac{t_0^{s-2}}{s-2} - \frac{6}{\pi}\frac{t_0^{s-1}}{s-1}\]
	is independent of $t_0$, and is the completed $L$-function corresponding to the Dirichlet series \[\sum_{n\ge 1} \frac{j_{N,n}(i\infty)/24}{n^s}.\]
	Explicitly, we have \[L(s) = \frac{\Gamma(s)}{(2\pi)^s}\cdot 24\zeta(s)\zeta(s-1) = \frac{24\Gamma(s)}{(2\pi)^s} \sum_{n\ge 1} \frac{\sigma(n)}{n^s}\]
	where, in the notation above, the archimedean factor is $L_\infty(s) = \frac{24\Gamma(s)}{(2\pi)^s}$, and the remaining local factors have the Euler product \[L^\infty(s) = \sum_{n\ge 1} \frac{\sigma(n)}{n^s} = \prod_p \frac1{1 - p^{-s}} \frac1{1 - p^{1-s}}.\]
	Moreover, it can be easily shown that this $L$-function satisfies 
	\begin{equation}\label{eqn::func_eqn_N=1_L-function_Eisenstein_series}
		L(s) = -L(2-s)
	\end{equation}
	For instance see \cite[][Lemma 3.1]{bringmann_kane_N1_case}.  
	
	We generalize the above results for $H_{1,i\infty}^*$ to $N\ge 2$ in the proposition below.
	
	\begin{proposition}\label{prop::mellin_transform_at_cusps}
		Fix $N\ge 2$. Define \[\begin{split}
			L_{N, i\infty}(s) := \frac{t_0^s}{s} - \frac{6}{[SL_2(\Z):\Gamma_0(N)]\pi} \frac{t_0^{s-1}}{s-1}\;+\;& \int_{0}^{t_0} t^{s-1} \left(H_{N,i\infty }^*(it) + \frac{3}{[SL_2(\Z):\Gamma_0(N)]\pi t} \right)dt\\
			& + \int_{t_0}^{\infty} t^{s-1} \left(H_{N,i\infty }^*(it) + 1 - \frac{3}{[SL_2(\Z):\Gamma_0(N)]\pi t} \right)dt
		\end{split}\]
		and 
		\[\begin{split}
			L_{N, 0}(s) := \frac1N \frac{t_0^{s-2}}{s-2} - \frac{6}{[SL_2(\Z):\Gamma_0(N)] \pi} \frac{t_0^{s-1}}{s-1} \;+\; &\int_{0}^{t_0} t^{s-1} \left(H_{N,0 }^*(it) - \frac1{Nt^2} + \frac{3}{[SL_2(\Z):\Gamma_0(N)]\pi t} \right)dt\\
			& + \int_{t_0}^{\infty} t^{s-1} \left(H_{N,0}^*(it)  - \frac{3}{[SL_2(\Z):\Gamma_0(N)]\pi t} \right)dt.
		\end{split}\]
		Then,
		\begin{enumerate}
			\item Both these functions are well-defined and independent of $t_0$. Moreover, $L_{N, i\infty}(s)$ is holomorphic on $\C\bs \{0,1\}$ and $L_{N, 0}(s)$ holomorphic on $\C\bs \{1,2\}$.
			\item For $\rho \in \{0, i\infty\}$, the function $L_{N, \rho}(s)$ provides a meromorphic continuation to $\C$ for \[\frac{\Gamma(s)}{(2\pi)^s}\sum_{n\ge 1} \frac{j_{N,n}(\rho)}{n^s},\]
			which \textit{a priori} was only well-defined for $\Re s>\frac52$.
			\item We have the functional equation $N^{s/2}L_{N, i\infty}(s) = - N^{(2-s)/2}L_{N, 0}(2-s)$.
			\item We have the explicit expressions 
			\begin{align*}
				L_{N, i\infty}(s) &= 24(2\pi)^{-s}\Gamma(s) \zeta(s) \zeta(s-1) \cdot N^{-s} \prod_{p|N} \frac{1 - p^{s-2}}{1-p^{-2}}\\
				L_{N, 0}(s) &= 24(2\pi)^{-s}\Gamma(s) \zeta(s) \zeta(s-1) \cdot \frac1N \prod_{p|N} \frac{1 - p^{-s}}{1-p^{-2}}
			\end{align*}
			the product being taken over all primes $p$ dividing $N$.
		\end{enumerate}
	\end{proposition}
	From the perspective of $L$-functions, it is easy to see that $N^{s/2}L_{N, \rho}$ ($\rho\in \{0,i\infty\})$ is the completed $L$-function corresponding to the Dirichlet series \[\sum_{n\ge 1} \frac{j_{N,n}(\rho)/24}{n^s},\]
	where the archimedean factor is given by $L_{N, \rho; \infty}(s) = 24N^{s/2}(2\pi)^{-s} \Gamma(s)$, and the local factors at finite places are given by 
	\begin{align*}
		L_{N, i\infty; p}(s) &= \begin{cases}
			1/(1 - p^{-s})(1 - p^{1-s}) & \text{if } p\nmid N,\\
			p^{-s \;\mr{ord}_pN}(1-p^{s-2})/(1-p^{-2})(1 - p^{-s})(1 - p^{1-s}) & \text{if } p|N.
		\end{cases}\\
		L_{N, 0; p}(s) &= \begin{cases}
			1/(1 - p^{-s})(1 - p^{1-s}) & \text{if } p\nmid N,\\
			p^{-\mr{ord}_pN}/(1-p^{-2})(1 - p^{1-s}) & \text{if } p|N.
		\end{cases}
	\end{align*}
	where $\mr{ord}_pN$ is the largest $k\in \N$ such that $p^k|N$. 

	The parallels between the above proposition and the case for $N=1$ \cite[c.f.][Lemma 3.1]{bringmann_kane_N1_case} are obvious---setting $N=1$ and using the fact that $H_{1,0}^* = H_{1,i\infty}^*$ essentially yields Lemma 3.1 of \cite{bringmann_kane_N1_case}, though the regularization of the Mellin transforms in $L_{1, i\infty}$ and $L_{1,0}$ need to be fixed to account for the fact that the Fourier expansion at $0$ is the same as that at $i\infty$. The differences are also noticeable; all of these differences are due entirely to the fact that $S = \smalltwosqmat0{-1}10\in SL_2(\Z)$ does not belong to $\Gamma_0(N)$ for any $N\ge 2$. Thus, the substitution $t\mapsto \frac{-1}{t}$ in the integral does not take $H_{N,i\infty}^*(it)$ to itself, but instead to $H_{N, 0}^*(it)$. Since the cusp $0$ and the cusp $i\infty$ are distinct cusps of $\Gamma_0(N)$ for $N\ge 2$, they yield distinct Dirichlet series, distinct $L$-functions, and so on.  
	
	The explicit expressions for the local factors $L_{N, \rho; p}$ also yield expressions for $j_{N,n}(\rho)$, as given in the following corollary.
	\begin{corollary}\label{cor::explicit_formula_for_j_N,n_rho}
		For any $n\in \N$, we have \[j_{N,n}(i\infty) = \frac{24}{N^2} \left( \prod_{p|N} \frac{1}{1 - p^{-2}}\right) \mu\left(\frac{N}{g}\right) g^2 \sum_{d|\gcd(N, n/g)} \frac{\mu(d)}{d^2} \sigma \left(\frac{nd}{g} \right)\]
		where $g:=\gcd(n,N)$, and $\sigma$ is the usual sum of divisors function. Also,
		\[j_{N,n}(0) = \frac{24}{N} \left( \prod_{p|N} \frac{1}{1 - p^{-2}}\right) \cdot n_2 \sigma(n_1)\]
		where we write $n=n_1n_2$ with $\gcd(n_1, N)=1=\gcd(n_1, n_2)$ and every prime factor of $n_2$ divides $N$.
	\end{corollary}
	As an aside, it is interesting to note that there is a constant factor of \[\frac{24}{N^2} \prod_{p|N} \frac{1}{1-p^{-2}} = \frac{24}{[SL_2(\Z):\Gamma_1(N)]}\]
	in the above expressions for $j_{N,n}(\rho)$, where $\Gamma_1(N)$ is the congruence subgroup \[\Gamma_1(N) = \left\{\smalltwosqmat{a}{b}{c}{d} \in SL_2(\Z): N|c, a,d\equiv 1\modc{N}\right\}.\]
	
	Next, we consider the Dirichlet series obtained from the Fourier expansion of $H_{N,z}^*(\tau)$, i.e. the formal series \[\sum_{n\ge 1} \frac{j_{N,n}(z)}{n^s} .\]
	However $j_{N,n}(z)$ grows like $e^{2\pi n \Im z}$ \cite[Theorem 1.1(4)]{divisors_mod_forms}, and so the above series does not converge for any $s\in \C$. Thus, it is meaningless to talk about the Dirichlet series of $H_{N,z}^*(\tau)$. On the other hand, we can still consider the Mellin transform of $H_{N,z}^*(\tau)$, as long as we suitably regularize it. This is the content of the next proposition.
	
	\begin{proposition}\label{prop::mellin_transform_z_converges_infty_only}
		Write $z = x+iy$, and suppose $z$ is not in the set \[\mc S_N := \{\mf z \in \mbb H : \exists \gamma \in \Gamma_0(N), \gamma \mf z \in i\R^+\} = \Gamma_0(N) \cdot (i\R^+).\] Suppose also that $y>2$. For any $N\ge 2$ and any $t_0>0$, the following integral exists, is independent of $t_0$, and is a meromorphic function of $s$ on $\C$ whose only pole is a simple pole at $s=1$:
		\[\begin{split}
			L_{N,z}(s) := -\frac{6}{[SL_2(\Z) : \Gamma_0(N)] \pi} \frac{t_0^{s-1}}{s-1} \;+ \;& \int_0^{t_0} t^{s-1} \left(H_{N,z}^*(it)  + \frac{3}{[SL_2(\Z) : \Gamma_0(N)] \pi t}\right)dt \\
			&+ \;\int_{t_0}^\infty t^{s-1}\left(H_{N,z}^*(it)  - \frac{3}{[SL_2(\Z) : \Gamma_0(N)] \pi t}\right)dt.
		\end{split}\]
	\end{proposition}
	Here, it is necessary that $z\notin \Gamma_0(N) (i\R^+)$ as otherwise there would be a simple pole on the line of integration. We also require $y>2$ in order to ensure convergence of the Fourier expansion of $H_{N,z}^*(\tau)$ at the cusps. However, from the Poincar\'e series definition of $H_{N,z}^*(\tau)$, it is not apparent that this restriction on $y$ is really necessary, and it would be interesting to check whether the above integral does indeed converge for all $y$.

	It should be noted that the Mellin transform (after removing suitable powers of $2\pi$ and $\Gamma$ factors) also does not seem to have an obvious Euler product, nor is there any reason for there to be an Euler product expansion. As there is neither a series representation nor an Euler product, it is not an $L$-function in the classical sense. However, it is still a meromorphic function on $\C$ satisfying a functional equation.
	
	\begin{proposition}\label{prop::mellin_transform_z_converges}
		Write $z = x+iy$, with the same conditions on $z$ as in Proposition \ref{prop::mellin_transform_z_converges_infty_only}. 
		\begin{enumerate}
			\item For any $N\ge 2$ and $y>2$, the following integral exists, is independent of $t_0$, and is a meromorphic function of $s$ on $\C$ whose only pole is a simple pole at $s=1$:
			\[\begin{split}
				L_{N,-1/Nz}(s) = -\frac{6}{[SL_2(\Z) : \Gamma_0(N)]\pi }\frac{t_0^{s-1}}{s-1} \;+&\;\int_0^{t_0} t^{s-1} \left(H_{N,-1/Nz}^*(it)  + \frac{3}{[SL_2(\Z) : \Gamma_0(N)] \pi t}\right)dt \\
				&+ \;\int_{t_0}^\infty t^{s-1}\left(H_{N,-1/Nz}^*(it)  - \frac{3}{[SL_2(\Z) : \Gamma_0(N)] \pi t}\right)dt.
			\end{split}\]
			\item We have the functional equation \[N^{s/2} L_{N,z}(s) = - N^{(2-s)/2} L_{N,-1/Nz}(2-s).\]
		\end{enumerate}
	\end{proposition}
	
	Bringmann and Kane \cite{bringmann_kane_N1_case} studied the $N=1$ case of Propositions \ref{prop::mellin_transform_z_converges_infty_only} and \ref{prop::mellin_transform_z_converges}, obtaining the same result as above \cite[Theorem 1.1]{bringmann_kane_N1_case}; as before since $S = \smalltwosqmat0{-1}10\in SL_2(\Z)$ we have $H_{1,-1/z}^* = H_{1,z}^*$ and so $L_{1,-1/z} = L_{1,z}$. However, they proved Proposition \ref{prop::mellin_transform_z_converges_infty_only} and \ref{prop::mellin_transform_z_converges} for $N=1$ without the constraint $y>2$. 
	
	Finally, since $H_{N,z}^* \to H_{N,\rho}^*$ as $z\to \rho$ for any cusp $\rho$ of $\Gamma_0(N)$, one would hope that $L_{N,z} \to L_{N, i\infty}$ as $z\to i\infty$ and $L_{N, z}\to L_{N, 0}$ as $z\to 0$ (or, $L_{N, -1/Nz}\to L_{N,0}$ as $z\to i\infty$). However, this is not the case. Due to a certain geometric series appearing in the Fourier expansion of $L_{N,z}(s)$ at $i\infty$, in order to make this limit converge it is necessary to subtract some terms. For $N=1$, Bringmann and Kane \cite[Theorem 4.3]{bringmann_kane_N1_case} showed that 
	\begin{equation}\label{eqn::limit_of_mellin_N=1}
		\begin{split}
			\lim_{y\to \infty} \Biggl( L_{1,x+iy}(s) &+ \frac{y^s}{s} - \sum_{j=1}^{\lfloor \Re{s} \rfloor} \frac{(s-1)_{j-1} y^{s-j}}{(2\pi)^j} \left(\mr{Li}_j\big(e^{-2\pi i x}\big) + (-1)^j \mr{Li}_j\big(e^{2\pi i x} \big)\right) \\
			& - \frac{y^{2-s}}{2-s} + \sum_{j=1}^{\lfloor 2-\Re{s} \rfloor } \frac{(1-s)_{j-1} y^{2-s-j}}{(2\pi)^j} \left(\mr{Li}_j\big(e^{-2\pi i x}\big) + (-1)^j \mr{Li}_j\big(e^{2\pi i x} \big)\right)\Biggl) \\
			&= L_{1, i\infty}(s) = \frac{24 \Gamma(s) \zeta(s) \zeta(s-1)}{(2\pi)^s},
		\end{split}
	\end{equation}
	where,
	\begin{itemize}
		\item $\lfloor a\rfloor$ the greatest integer less than or equal to $a\in \R$;
		\item $a_j = a(a-1) \cdots (a-j+1)$ is the falling factorial;
		\item the empty sum is defined to be zero; and
		\item $\mr{Li}_p(z)$ is the polylogarithm, defined initially for arbitrary $p\in \C$ and any $|z|<1$ by \[\mr{Li}_p(z) = \sum_{n\ge 1} \frac{z^n}{n^p},\]
		which is then analytically continued to $|z|\ge 1$ ($z\notin \R^+$) by the integral representation \cite[][{}(1.2)]{wood_polylogs} \[\mr{Li}_p(z) = \frac{1}{\Gamma(p)} \int_0^\infty \frac{t^{p-1}}{e^t/z - 1} dt.\]
	\end{itemize}
	In their Theorem 4.3, their `$L_z(s)$' is in fact equal to $-\frac{1}{2\pi i} L_{1, z}(s)$ in our notation.
	
	In this paper, we generalize the limit in (\ref{eqn::limit_of_mellin_N=1}) to arbitrary $N\ge 2$. Notice the difference between the result for $N=1$ and for $N\ge 2$.
	
	\begin{theorem}\label{thm::mellin_transform_converge_limit}
		Write $z = x+iy\notin \mc S_N$. For all $s\in \C$ and $N\ge 2$,
		\[\lim_{y\to \infty} \Biggl(L_{N,z}(s) + \frac{y^s}{s} - \sum_{j=1}^{\lfloor \Re{s} \rfloor} \frac{(s-1)_{j-1} y^{s-j}}{(2\pi)^j} \Big(\mr{Li}_j\big(e^{-2\pi i x}\big) + (-1)^j\mr{Li}_j\big( e^{2\pi i x}\big) \Big) \Biggl) = L_{N, i\infty}(s),\]
		where the empty sum is zero.
	\end{theorem}
	The reason the polylogarithm terms are well-defined is that $z\notin \mc S_N$, and so in particular $x\notin \Z$. Proposition \ref{prop::mellin_transform_z_converges}(2) and \ref{prop::mellin_transform_at_cusps}(3) also immediately yield the following corollary.
	\begin{corollary}\label{cor::mellin_transform_at_0_converge_limit}
		Write $z = x+iy\notin \mc S_N$. For all $s\in \C$ and $N\ge 2$,
		\[\lim_{y\to \infty} \Biggl(L_{N,-1/Nz}(s) - \frac{N^{1-s} y^{2-s}}{2-s} + N^{1-s}\sum_{j=1}^{\lfloor 2-\Re{s} \rfloor} \frac{(1-s)_{j-1} y^{2-s-j}}{(2\pi)^j} \Big(\mr{Li}_j\big(e^{-2\pi i x}\big) + (-1)^j\mr{Li}_j\big( e^{2\pi i x}\big) \Big) \Biggl) = L_{N, 0}(s).\]
	\end{corollary}
	The reason the polylogarithm terms in the case of $N=1$ and that for $N\ge 2$ are different is that the Fourier expansion (given in Proposition \ref{prop::fourier_at_arb_cusps}) at a cusp not $\Gamma_0(N)$-equivalent to $i\infty$ does not have the problematic geometric series term, and so there is no contribution of polylogarithm terms near 0. On the other hand, notice that the sum of the two limits for $N\ge2$ is of the same shape, and has the same polylogarithm terms, as the limit for $N=1$.
	
	In proving the $N=1$ version of both Propositions \ref{prop::mellin_transform_z_converges_infty_only} and \ref{prop::mellin_transform_z_converges}, as well as equation \ref{eqn::limit_of_mellin_N=1}, Bringmann and Kane \cite{bringmann_kane_N1_case} expressed $H_{1,z}^*$ in terms of the analytic continuation to $w=1$ of a certain derivative $\mc G_w(z, \tau)$ of the \emph{resolvent kernel} for $SL_2(\Z)$ (c.f. \cite{Gross-Zagier_resolvent_kernel}). They then analyzed the (regularized) Mellin transform of $\mc G_w$, which they then analytically continued to $w=1$ to obtain $L_{1,z}$. One could theoretically do something similar---express $H_{N,z}^*$ in terms of the analytic continuation of a certain derivative of the resolvent kernel for $\Gamma_0(N)$ (given in \cite[][{}(2.10)]{Gross-Zagier_resolvent_kernel}), and then analyze some regularization of the Mellin transform of the resolvent kernel. We instead go a different more elementary route. In this paper we directly analyze the Fourier expansion of $H_{N,z}^*(it)$ for $t$ in different intervals. 
	
	The paper proceeds as follows. In Section \ref{sctn::prelim}, we recall some basic facts about polar harmonic Maass forms, the family of functions $H_{N,z}^*(\tau)$, and the weight 2 Eisenstein series for $\Gamma_0(N)$. In Section \ref{sctn::fourier_exp_arb_cusps} we compute the Fourier expansion of $H_{N,z}^*(\tau)$ (for fixed $z$) at arbitrary cusps of $\Gamma_0(N)$; the proof is split up into two subsections. Section \ref{sctn::useful_lemmas} proves some useful lemmas on a certain integral of various pieces of the Fourier expansion at $i\infty$ and $0$, and also gives the proof for Propositions \ref{prop::mellin_transform_z_converges_infty_only} and \ref{prop::mellin_transform_z_converges}. The lemmas in this section also provide most of the technical analysis required for Proposition \ref{prop::mellin_transform_at_cusps} and Theorem \ref{thm::mellin_transform_converge_limit}. We then prove Proposition \ref{prop::mellin_transform_at_cusps} (and its corollary) in Section \ref{sctn::continue_series_jNm_infty}. Finally, Section \ref{sctn::proof_thm} proves Theorem \ref{thm::mellin_transform_converge_limit}.
	
	\subsection*{Acknowledgements}
	The author would like to thank Ben Kane for all of his help and support during this project. This project was funded by the Wong Shek-Yung Memorial Fund, as part of the Summer Research Programme of the Graduate School, The University of Hong Kong.
	
	\section{Preliminaries}\label{sctn::prelim}
	
	\subsection{Polar Harmonic Maass Forms}\label{sctn::prelim_phmf}
	Let us briefly recall the definition of polar harmonic Maass forms \cite[c.f.][{}13.3]{hmf_and_mmf}. First, for any $X \in GL_2(\Q^+)$ and $k\in \Z$ we define the weight-k \emph{slash} operator on any function $f:\mbb{H}\to \C$ by \[f|_k[X](\tau) := (\det X)^{k/2} j(X, \tau)^{-k} f(X\tau)\]
	where, for $X= \smalltwosqmat abcd$, we define  \[X\tau := \frac{a \tau + b}{c\tau + d} \quad \text{ and } \quad j(X, \tau) = c\tau + d.\]
	We can then define polar harmonic Maass forms for integral weight as follows.
	\begin{definition}
		For $k\in \Z$ and a congruence subgroup $\Gamma$, a \emph{polar harmonic Maass Form of weight $k$ on $\Gamma$} is a function $F:\mbb{H}\to \C$ that is real-analytic outside a discrete set of points and satisfies 
		\begin{itemize}
			\item \emph{modularity for weight $k$ and $\Gamma$}, i.e. $F|_k[M] = F$ for all $M\in \Gamma$;
			\item the differential equation $\Delta_kF = 0$ where $\Delta_k$ is the \emph{weight $k$ hyperbolic Laplacian} \[\Delta_k := -y^2 \left( \hipdiff{2}{}{x} + \hipdiff{2}{}{y} \right) + iky \left( \pdiff{}{x} + i \pdiff{}{y} \right) ;\]
			\item For all $a \in \mbb H$, there exists $n\in \N_0$ such that $(\tau - a)^n F(\tau)$ is bounded in a neighborhood of $a$;
			\item The function $F$ grows at most linear exponentially towards cusps of $\Gamma_0(N)$.
		\end{itemize}
		The space of polar harmonic Maass forms is denoted by $\mc H_k(\Gamma)$.
	\end{definition}
	Similar definitions hold for half-integral weight, though we will not be needing them here. Throughout this paper, we will only be considering polar harmonic Maass forms for the congruence subgroup $$\Gamma_0(N) := \{\smalltwosqmat abcd \in SL_2(\Z): N|c\}$$ for $N\in \N$. We recall the standard fact about the index of this subgroup in $SL_2(\Z)$: \[[SL_2(\Z) : \Gamma_0(N)] = N \prod_{p|N}\left(1 + \frac1p \right)\]
	where the product runs over all prime factors of $N$.
	
	\subsection{$H_{N,z}^*(\tau)$ for $z\in \mbb H$}\label{sctn::prelim_HNz}
	We consider the following Poincar\'e series which was introduced in \cite{construction_Maass_form} \[P_{N,s}(\tau, z) = \sum_{M\in \Gamma_0(N)} \frac{\varphi_s(M\tau, z)}{j(M, \tau)^2 |j(M,\tau)|^{2s}} \]
	where \[\varphi_s(\tau, z) := \frac{\Im{z}^{1+s}}{(\tau - z) (\tau - \bar z) |\tau - \bar z|^{2s}}.\]
	In \cite{construction_Maass_form}, it was shown that this Poincar\'e series has an analytic continuation to $s=0$ denoted by $\Im{z} \Psi_{2,N}(\tau, z)$ (this notation follows that of Petersson \cite{petersson_poincare_series}). 
	We can then define the weight 2 polar harmonic Maass form for $\Gamma_0(N)$
	\begin{equation}\label{eqn::H_Nz_defn}
		H_{N,z}^*(\tau) := -\frac{1}{2\pi } \Im{z} \Psi_{2,N}(\tau, z) \in \mc H_2(\Gamma_0(N)).
	\end{equation}
	
	An explicit expression for $H_{N,z}^*(\tau)$, including its Fourier expansion at $i\infty$, was computed in \cite{construction_Maass_form} and \cite{divisors_mod_forms}. If $\tau$ is not $\Gamma_0(N)$-equivalent to $z$, then \[-2\pi H_{N,z}^*(\tau) = \Im{z}\Psi_{2,N}(\tau, z) = \Sigma_1 + \Sigma_2 + \Sigma_3\]
	where $\Sigma_1$, $\Sigma_2$, and $\Sigma_3$ are functions of $\tau$ (for fixed $z$). To write out these functions explicitly, consider the \emph{Kloosterman sum}
	\begin{equation}\label{eqn::basic_kloosterman_sum_defn}
		K(m,n;c) := \sum_{\begin{smallmatrix} a, d \modc{c} \\ ad \equiv 1 \modc{c} \end{smallmatrix}} e\left(\frac{md+na}c\right)
	\end{equation}
	where for $w\in \C$ we use the shorthand $e(w) := e^{2\pi i w}$. More generally, for a cusp $\rho$ of $\Gamma_0(N)$, let $M_\rho\in SL_2(\Z)$ be such that $M_\rho (i\infty) = \rho$, and let $\ell_\rho$ be the \emph{cusp width} (i.e. smallest $\ell\in \N$ such that $\smalltwosqmat 1\ell 01 \in M_\rho^{-1}\Gamma_0(N) M_\rho$); then we define the generalized Kloosterman sum
	\begin{equation}\label{eqn::general_kloosterman_sum_defn}
		K_{i \infty, \rho}(m,n;c) := \sum_{ \smalltwosqmat abcd \in \Gamma_\infty \bs \Gamma_0(N) M_\rho / \Gamma_\infty^{\ell_\rho}} e\left(\frac{md}{\ell_\rho c} + \frac{na}c \right)
	\end{equation}
	where $\Gamma_\infty := \{\pm \smalltwosqmat 1n01: n\in \Z\}$. In particular, \[K_{i\infty, i\infty}(m,n;c) = \begin{cases}
		K(m,n;c) & \text{if } N|c,\\
		0 & \text{if } N\nmid c.
	\end{cases}\]
	We will also need to make use of the $I$- and $J$-Bessel functions, which we denote as usual by $I_k$ and $J_k$ respectively. In particular, we use the $I_1$ and $J_1$ Bessel functions given by \[J_1(x) = \sum_{m\ge 0} \frac{(-1)^m}{m!(m+1)!} \left( \frac{x}{2}\right)^{2m+1} \quad \text{ and } \quad I_1(x) = \frac1i J_1(ix).\]
	
	Now, in \cite{construction_Maass_form} it was calculated that \begin{equation}\label{eqn::infty_sigma3_full_expression}
		\begin{split}
			\Sigma_3 ={}& -\frac{6}{[SL_2(\Z):\Gamma_0(N)] \Im{\tau}} - 8\pi^3 \sum_{m\ge 1} m\sum_{c\ge 1} \frac{K_{i \infty, i\infty}(0,n;c)}{c^2}e^{2\pi i m \tau} \\
			&\quad - 8\pi^3 \sum_{m\ge 1} m e^{2\pi i m \tau} \left(\sum_{n\le -1} \sum_{c\ge 1} \frac{K_{i \infty, i\infty}(m,n;c)}{c^2} e^{-2\pi i n z}  + \sum_{n\ge 1} \sum_{c\ge 1} \frac{K_{i \infty, i\infty}(m,n;c)}{c^2} e^{-2\pi i n \bar z}\right)
		\end{split}
	\end{equation}
	which converges for all $\tau $ and $z$. Here, the first sum is just \[-2\pi \sum_{m\ge 1} j_{N,m}(i \infty) e^{2\pi i m \tau}.\]
	From \cite{construction_Maass_form} (for $\Im{\tau}<\Im z$) and \cite{divisors_mod_forms} (for $\Im{\tau}>\Im z$) we also have
	\begin{align}
		\Sigma_1 &= 2\pi \cdot \begin{cases}
			\sum_{n\le 0} e^{-2\pi i n \tau} e^{2\pi i n z} + \sum_{n\ge 1} e^{2\pi i n \tau} e^{-2\pi i n \bar z} & \text{if } \Im{\tau} < \Im{z} \\
			-\sum_{n\ge 1} \left( e^{-2\pi i n z} - e^{-2\pi i n \bar z}\right)  e^{2\pi i n \tau} & \text{if } \Im{\tau} > \Im{z} \\
		\end{cases} \nonumber\\
		&= \frac{-2\pi}{e^{2\pi i (\tau - z)} - 1} + 2\pi \sum_{n\ge 1} e^{-2\pi i n \bar z} e^{2\pi i n \tau} \label{eqn::infty_sigma1_full_expression}
	\end{align}
	which holds for all $\tau \in \mbb H\bs \{z\}$. Finally, from \cite{construction_Maass_form} we have 
	\begin{equation}\label{eqn::infty_sigma2_full_expression}
		\begin{split}
			\Sigma_2 ={}& 8\pi^3 \sum_{m\ge 1} m e^{2\pi i m \tau} \left(\sum_{n\le -1} \sum_{c\ge 1} \frac{K_{i \infty, i\infty}(m,n;c)}{c^2} e^{-2\pi i n z}  + \sum_{n\ge 1} \sum_{c\ge 1} \frac{K_{i \infty, i\infty}(m,n;c)}{c^2} e^{-2\pi i n \bar z}\right) \\
			&\quad + 4\pi^2 \sum_{m\ge 1} \sum_{n,c\ge 1} \sqrt{\frac mn} \frac{K_{i\infty, i\infty}(m,-n;c)}{c} I_1 \left(\frac{4 \pi \sqrt{mn}}c \right) e^{2\pi i n z} e^{2\pi i m \tau}\\
			& \quad + 4\pi^2 \sum_{m\ge 1} \sum_{n,c\ge 1} \sqrt{\frac mn} \frac{K_{i\infty, i\infty}(m,n;c)}{c} J_1 \left(\frac{4 \pi \sqrt{mn}}c \right) e^{-2\pi i n \bar z} e^{2\pi i m \tau},
		\end{split}
	\end{equation}
	which originally holds for $\Im{\tau} > \max\{\Im{z}, 1/\Im z \}$. However, it will be shown in Lemma \ref{lem::I2_convergence} that for large enough $\Im{z}$, this series in fact converges nicely for $\Im \tau > 1/\Im z$. This is because the residue of $H_{N,z}^*$ at the poles $\Gamma_0(N)z$ is $\frac{1}{4\pi i} \#\mr{Stab}_z(\Gamma_0(N))$, which is precisely $\frac{1}{2\pi i}$ if $z$ is not an elliptic fixed point (i.e. $z\notin \Gamma_0(N)i \cup \Gamma_0(N) \omega$ where $\omega = e^{2\pi i/3}$). However, it is easy to see that the residue of $(-\frac1{2\pi}) \Sigma_1$ at the poles $\Gamma_0(N)z$ is precisely $\frac{1}{2\pi i}$, which implies that $\Sigma_2$ is actually holomorphic on $\Gamma_0(N)z$ for $z$ not an elliptic fixed point. Thus, one would correctly expect that $\Sigma_2$ should also converge at $\tau = z$, and more generally at $\Im \tau = \Im z$. 
	
	On the other hand, the condition $\Im \tau > 1/\Im z$ is still necessary for convergence. This is because $I_1(x) = \Theta(x^{-1/2} e^x)$ by \cite[][{}(9.7.1)]{handbook_fns}, where by $f = \Theta(g)$ we mean that there exists $c,C>0$ such that $cg<f<Cg$. This implies that for large enough $m$ and $n$, and for $c=1$, the expression being summed is essentially some polynomial expression in $m$ and $n$ times the exponential factor \[\exp\left(2\pi (2\sqrt{mn} - ny - m v)\right)\]
	where $y = \Im z $ and $v = \Im \tau$. If $yv \le 1$, and if $m \approx \frac yvn$, then \[\exp\left(2\pi (2\sqrt{mn} - ny - m v)\right) \approx \exp\left(4\pi n \sqrt{\frac yv}(1 - \sqrt{yv}) \right)\]
	and the series over $n$ would diverge. 
	
	
	Combining the expressions for $\Sigma_1$, $\Sigma_2$, and $\Sigma_3$ gives us the Fourier expansion of $H_{N,z}^*(\tau)$ at $i\infty$, 
	\begin{proposition}\label{prop::fourier_expansion_infty}
		For $\Im{\tau} > \max\{\Im{z}, \frac{1}{\Im{z}}\}$, we have 
		\[\begin{split}
			H_{N,z}^*(\tau) ={}& \frac{3}{[SL_2(\Z): \Gamma_0(N)] \pi \Im{\tau}} + \sum_{m\ge 1} \left( e^{-2\pi i m z} - e^{-2\pi i m \bar z}\right) e^{2\pi i m \tau} \\
			& \quad + 2\pi \sum_{m\ge 1} \sum_{n,c\ge 1} \sqrt{\frac mn} \frac{K_{i\infty, i\infty}(m,-n;c)}{c} I_1 \left(\frac{4 \pi \sqrt{mn}}c \right) e^{2\pi i n z} e^{2\pi i m \tau}\\
			& \quad + 2\pi \sum_{m\ge 1} \sum_{n,c\ge 1} \sqrt{\frac mn} \frac{K_{i\infty, i\infty}(m,n;c)}{c} J_1 \left(\frac{4 \pi \sqrt{mn}}c \right) e^{-2\pi i n \bar z} e^{2\pi i m \tau} \\
			&\quad + 4\pi^2 \sum_{m\ge 1} m\left(\sum_{c\ge 1} \frac{K_{i\infty, \rho}(m,0;c)}{c^2} \right) e^{2\pi i m \tau}
		\end{split}.\]
	\end{proposition}
	However, as mentioned, if we assume that $\Im z$ is large enough, the above Fourier expansion holds for all $\Im \tau > 1/\Im z$, as long as we replace the geometric series in $\tau - z$ appropriately.
	
	\subsection{Weight 2 Eisenstein Series for $\Gamma_0(N)$}\label{sctn::eisenstein_series}
	
	For a cusp $\rho$ of $\Gamma_0(N)$, we can construct the harmonic weight 2 Eisenstein series $E_{2,N,\rho}^*(\tau)$ for $\Gamma_0(N)$ by analytically continuing the Poincar\'e series (\textit{a priori} defined only for $\Re{s}>0$)
	\[E_{2,N,\rho,s}^*(\tau) := \sum_{M\in \Gamma_\rho \bs \Gamma_0(N)} j(M_\rho M, \tau)^{-2} |j(M\rho M, \tau)|^{-2s}\]
	to $s=0$, where $M_\rho \in SL_2(\Z)$ is such that $M_\rho(i\infty) = \rho$, and $\Gamma_\rho = M_\rho \Gamma_\infty M_\rho^{-1}$ is the stabilizer of $\rho$ in $\Gamma_0(N)$ with $\Gamma_\infty = \{\pm \smalltwosqmat1n01\}$ \cite{analytic_continue_eisenstein_hecke}. By Theorem 1.2 of \cite{divisors_mod_forms}, we have 
	\[-E_{2,N,\rho}^*(\tau) = H_{N,\rho}^*(\tau) := \lim_{z\to \rho} H_{N,z}^*(\tau) = \frac{3}{[SL_2(\Z):\Gamma_0(N)]\pi \Im{\tau}} - \delta_{\rho, i\infty} + \sum_{m\ge 1} j_{N,m}(\rho) e^{2\pi i m \tau}\]
	where $\delta_{\rho, i\infty}=1$ if $\rho \in \Gamma_0(N) \cdot (i\infty)$ and $=0$ otherwise, and \[j_{N,n}(\rho) = \frac{4\pi^2 n }{\ell_\rho} \sum_{c\ge 1} \frac{K_{i\infty, \rho}(n,0;c)}{c^2}.\]

	\section{The Fourier Expansion of $H_{N,z}^*(\tau)$ at arbitrary cusps}\label{sctn::fourier_exp_arb_cusps}

	Let us now compute the Fourier expansion at an arbitrary cusp $\rho = \alpha/\gamma$ not $\Gamma_0(N)$-equivalent to $i\infty$, where $\gamma|N$, and let $L = \smalltwosqmat{\alpha}{\beta} {\gamma} {\delta} \in SL_2(\Z)$ be such that $\rho = L(i\infty)$. Let $\ell = \ell_\rho$ be the width of the cusp, explicitly given by $$\ell = \frac{N / \gamma}{\gcd(N/\gamma, \gamma)}.$$
	
	The Fourier expansion at $\rho$ ($\ne g(i\infty)$ for $g\in \Gamma_0(N)$) is then given in the following proposition. 
	\begin{proposition}\label{prop::fourier_at_arb_cusps}
		With the above notation, for $\Im{\tau} > \max\{\Im{z}, 1/\Im{z}\}$ we have 
		\[\begin{split}
			H_{N,z}^*|_2[L](\tau) ={}& \frac{3}{[SL_2(\Z): \Gamma_0(N)] \pi \Im{\tau}} + \frac{4\pi^2}{\ell^2} \sum_{m\ge 1} m \left( \sum_{c\ge 1} \frac{K_{i\infty, \rho}(m,0;c)}{c^2} \right) e^{2\pi i m \tau / \ell} \\
			&\quad- \frac{2\pi}{\ell^{3/2}} \sum_{m\ge 1} \sqrt m \left(\sum_{n,c \ge 1} \frac{e^{-2\pi i n \bar z} K_{i \infty, \rho} (m, n; c)}{c \sqrt n} J_1 \left( \frac{4\pi}{c} \sqrt{\frac{mn}{\ell}} \right) \right) e^{2\pi i m \tau / \ell}\\
			&\quad - \frac{2\pi}{\ell^{3/2}} \sum_{m\ge 1} \sqrt m \left(\sum_{n,c \ge 1} \frac{e^{2\pi i n z} K_{i \infty, \rho} (m, -n; c)}{c \sqrt n} I_1 \left( \frac{4\pi}{c} \sqrt{\frac{mn}{\ell}} \right) \right) e^{2\pi i m \tau / \ell}.
		\end{split}\]
	\end{proposition}
	
	To compute the Fourier expansion, we use the same technique used in \cite{construction_Maass_form}---we split the original Poincar\'e series $P_{N,s}(\tau, z)$ into two separate sums, and then analytically continue each one to $s=0$. Using the identity $j(M, L\tau) j(L, \tau) = j(ML, \tau)$, we have
	\begin{align*}
		\frac{P_{N, s}(L\tau, z)}{j(L, \tau)^{2} |j(L, \tau)|^{2s}} &= \sum_{M\in \Gamma_0(N)} \frac{\varphi_s(ML\tau, z)}{j(ML, \tau)^2 |j(ML, \tau)|^{2s}} = \sum_{M\in \Gamma_0(N)L} \frac{\varphi_s(M\tau, z)}{j(M, \tau)^2 |j(M, \tau)|^{2s}} \\
		&= 2\;\;\labelunder{\Sigma_{N,\rho}^{(1)}(s)}{ \sum_{ \begin{smallmatrix} M = \smalltwosqmat abcd\in \Gamma_0(N) L \\ c\ge 1\end{smallmatrix}} \frac{\varphi_s(M \tau, z) - \varphi_s(\frac ac, z)}{j(M, \tau)^2 |j(M, \tau)|^{2s}} } + 2\;\;\labelunder{\Sigma_{N,\rho}^{(2)}(s)}{ \sum_{ \begin{smallmatrix} M  = \smalltwosqmat abcd\in \Gamma_0(N) L \\ c\ge 1\end{smallmatrix}} \frac{ \varphi_s(\frac ac, z)}{j(M, \tau)^2 |j(M, \tau)|^{2s}} }.
	\end{align*}
	The reason we don't have a third sum is that for $\rho$ not $\Gamma_0$-equivalent to $i\infty$, every matrix $\smalltwosqmat abcd\in \Gamma_0(N) L$ must satisfy $c\ne 0$. 
	
	It remains to analytically continue these sums to $s=0$, the computation for which has been split in the following two sections. The following computation proves useful:
	\begin{lemma}[Equation (4.3) of \cite{construction_Maass_form}] \label{lem::compute_g_n}
		For any $n\in \Z$, any $w_1 \in \C\bs \R$, and any $w_2 \in \mbb H$, define the integral \[g_n(w_1, w_2) := \int_\R \frac{e^{-2\pi i n t}}{(w_1 + t)(w_2 + t)}dt.\]
		Then \[g_n(w_1, w_2) = \begin{cases}
			0 & \text{if } n\le 0 \text{ and } w_1\in \mbb H,\\
			2\pi i (w_2 - w_1)^{-1} e^{2\pi i n w_1} & \text{if } n\le 0 \text{ and } w_1 \in -\mbb H,\\
			2\pi i (w_2 - w_1)^{-1} e^{2\pi i n w_2} & \text{if } n> 0 \text{ and } w_1 \in -\mbb H,\\
			2\pi i (w_2 - w_1)^{-1} \left(e^{2\pi i n w_2} - e^{2\pi i n w_1} \right) & \text{if } n> 0, w_1 \in \mbb H, \text{and } w_1\ne w_2,\\
			- 4\pi^2 n e^{2\pi i n w_1} & \text{if } n>0 \text{ and } w_1 = w_2.
		\end{cases}\]
	\end{lemma}
	This lemma can be proved by using the Residue Theorem on the square with vertices $-R$, $R$, $-R + i\epsilon R$ and $R+i\epsilon R$, and then taking the limit as $R\to \infty$. Here $\epsilon= -1$ if $n\le0$, and $=1$ if $n>0$.
	
	\subsection{Analytically continuing $\Sigma_{N, \rho}^{(2)}(s)$ to $s=0$}
	
	Let us first compute the analytic continuation of $\Sigma_{N, \rho}^{(2)}(s)$. 	
	In order to express the result, let $\mbb P(n)$ denote the set of prime factors of $n$. Split $\ell = \ell_1 \ell_2$ such that $\mbb P(\ell_1) = \mbb P(\ell) \cap \mbb P(\gamma)$ and $\gcd(\ell_2, \gamma) = 1 = \gcd(\ell_1, \ell_2)$. Then as $\ell | N/\gamma$, we can write $N/\gamma = N_1 \ell_2$. We can also write $\ell_1 \gamma = A_1A_2$ where $\mbb P(A_1) = \mbb P(\ell_1 \gamma) \cap \mbb P(N_1)$ and $\gcd(A_2, N_1) = 1 = \gcd(A_1, A_2)$. Let $\mu$ as usual denote the M\"obius function and $\phi$ Euler's phi function.
	\begin{lemma}\label{lem::analytic_cont_Sigma2_arb_s}
		The sum $\Sigma_{N, \rho}^{(2)}(s)$ can be analytically continued to $\Re s > -1/4$ with analytic continuation given by 
		\[\begin{split}
			- &\frac{\sqrt{\pi}}{1+s} \frac{\Gamma(\frac12 +s)}{ \Gamma(1+s)} \left( \frac\ell{\Im{\tau}} \right)^{1+2s} \left( \int_\R \varphi_s(w,z) dw\right) \frac{A_1}{N_1 \phi(\ell_2 A_1)} \frac{s\zeta(2s+1)}{\zeta(2s+2)} \sum_{g|N/\gamma} \mu(g) \frac{\phi(g \gamma \ell)}{(g \gamma \ell)^{2+2s}} \prod_{p| g \gamma \ell} \frac{1}{1 - p^{-2 - 2s}} \\
			& +\sum_{(m,n)\in \Z^2\bs\{(0,0)\}} (\Im{z})^{1+s} \left(\frac{1}{\ell^{2+2s}}\int_\R \frac{e^{-2\pi i m t}}{(\frac \tau \ell + t)^{2+s} (\frac{\bar \tau}{\ell} + t)^{s}} dt \right) \left(\int_\R \frac{e^{-2\pi i n w}}{(w- z)^{1+s}(w- \bar z)^{1+s}} dw\right) \sum_{c\ge 1} \frac{K_{i\infty, \rho}(m,n;c) }{c^{2+2s}}.
		\end{split}\]
	\end{lemma}
	\begin{proof}
		We decompose the sum into sums over cosets by writing $M = \smalltwosqmat1n01 M' \smalltwosqmat{1}{\ell m}{0}{1}$, 
		where $n$ and $m$ range over all $\Z$, and $M'$ ranges over the cosets in $\Gamma_\infty \bs \Gamma_0(N) L / \Gamma_\infty^{\ell}$. Applying Poisson summation twice, we then get 
		\begin{align*}
			\Sigma_{N, \rho}^{(2)}(s) 
			&= \sum_{c\ge 1} \frac{1}{c^{2+2s}} \sum_{M  = \smalltwosqmat abcd\in \Gamma_\infty \bs \Gamma_0(N) L / \Gamma_\infty^{\ell}} \sum_{m\in \Z} \sum_{n\in \Z} y^{1+s} \left(\int_\R \frac{e^{-2\pi i m t}}{(\tau + \frac dc + \ell t)^2 |\tau + \frac dc + \ell t|^{2s}} dt \right) \\ &\qquad \qquad  \left(\int_\R \frac{e^{-2\pi i n w}}{(\frac ac - z + w)(\frac ac - \bar z + w)|\frac ac - \bar z + w|^{2s}} dw\right)
		\end{align*}
		Shifting the integrals using the substitutions $\frac dc + \ell t\mapsto t$ and $\frac ac + w\mapsto w$ we have 
		\begin{align}
			\Sigma_{N, \rho}^{(2)}(s) &= \sum_{c\ge 1} \frac{1}{c^{2+2s}} \sum_{M  = \smalltwosqmat abcd\in \Gamma_\infty \bs \Gamma_0(N) L / \Gamma_\infty^{\ell}} \sum_{m\in \Z} \sum_{n\in \Z} y^{1+s} e(dm/\ell c) \left( \frac 1\ell \int_\R \frac{e^{-2\pi i m t/\ell}}{(\tau + t)^2 |\tau + t|^{2s}} dt \right) \\ &\qquad \qquad  e(an/c)\left(\int_\R \frac{e^{-2\pi i n w}}{(w- z)(w- \bar z)|w- \bar z|^{2s}}dw\right) \nonumber\\
			&= \sum_{m,n\in \Z} y^{1+s} \left(\frac{1}{\ell^{2+2s}}\int_\R \frac{e^{-2\pi i m t}}{(\tau/\ell + t)^{2+s} (\bar \tau/\ell + t)^{s}} dt \right) \left(\int_\R \frac{e^{-2\pi i n w}}{(w- z)^{1+s}(w- \bar z)^{1+s}} dw\right) \sum_{c\ge 1} \frac{K_{i\infty, \rho}(m,n;c) }{c^{2+2s}} \label{eqn::sigma_N_rho_(2)_partial_formula}
		\end{align}
		with $y:= \Im{z}$. \textit{A priori}, this identity holds only formally. We show absolute convergence, and uniform convergence for $s$ in right half-planes. Note that the proof of Lemma 3.2 in \cite{construction_Maass_form} (c.f. inequalities (3.7) and (3.8) of \cite{construction_Maass_form}) shows that 
		\[\left| \int_\R \frac{e^{-2\pi i m t}}{(\tau/\ell + t)^{2+s} (\bar \tau/\ell + t)^{s}} dt \right| \ll _{\tau, \ell, \sigma_0} e^{-\pi |m| \Im \tau}\]
		and 
		\[\left| \int_\R \frac{e^{-2\pi i n w}}{(w- z)^{1+s}(w- \bar z)^{1+s}} dw \right| \ll_{z, \sigma_0} e^{-\pi |n| y}\]
		for $-1/2 < \sigma_0 < \Re s$. Using these bounds as well as the bound on $K_{i\infty, \rho}(m,n;c)$ given in Lemma \ref{lem::weil_bound_general_kloostermansum} below, the sum over $(m,n)\ne (0,0)$ in (\ref{eqn::sigma_N_rho_(2)_partial_formula}) can be bounded against \[\ll_{z, \tau, \epsilon, \sigma_0}\sum_{c\ge 1} c^{-\frac32 + \epsilon - 2\sigma_0} \sum_{n, m\in \Z} \sqrt{|n|}e^{-\pi |n| y}e^{-\pi |m| \Im{\tau}} \]
		for any $\Re{s} > \sigma_0 > \frac{\epsilon}{2} -\frac14$ and for any $\epsilon>0$. Since the above series converges, the identity in (\ref{eqn::sigma_N_rho_(2)_partial_formula}) is valid, and moreover the sum over $(m,n)\ne (0,0)$ converges to an analytic function in $s$ in the region $\Re{s}>-\frac14$ for fixed $\tau$ and $z$. 
		
		All that remains is to analytically continue the $m=n=0$ term; this term is 
		\begin{equation}\label{eqn::fourierexp_sigma2_m=n=0}
			\left(\frac{1}{\ell^{2+2s}} \int_\R \frac{dt}{(\tau /\ell + t)^{2+s} ({\bar \tau }/\ell + t)^s}\right) \left(\int_\R \frac{y^{1+s}dw}{(w- z)^{1+s}(w- \bar z)^{1+s}}\right) \sum_{c\ge 1} \frac{\#\{\smalltwosqmat abcd \in \Gamma_\infty \bs \Gamma_0(N) L / \Gamma_\infty ^\ell \}}{c^{2+2s}}.
		\end{equation}
		Now, using Lemma \ref{lem::formula_for_K_infty,rho_0,0,c} below followed by the same calculation given in \cite[][Lemma 5.3]{construction_Maass_form}, we have \[\sum_{c\ge 1} \frac{\#\{\smalltwosqmat abcd \in \Gamma_\infty \bs \Gamma_0(N) L / \Gamma_\infty ^\ell \}}{c^{2+2s}} = \frac{A_1 \ell^{2+2s}}{N_1 \phi(\ell_2 A_1)} \frac{\zeta(2s+1)}{\zeta(2s+2)} \sum_{g|N/\gamma} \mu(g) \frac{\phi(g \gamma \ell)}{(g \gamma \ell)^{2+2s}} \prod_{p| g \gamma \ell} \frac{1}{1 - p^{-2 - 2s}}\]
		which gives a meromorphic continuation to $\C$.
		Using equation (3.10) from \cite{construction_Maass_form} for the term in the first bracket of equation (\ref{eqn::fourierexp_sigma2_m=n=0}), and recognizing that the integrand in the second bracket is just $\varphi_s(w,z)$ (\ref{eqn::fourierexp_sigma2_m=n=0}), the $m=n=0$ term is equal to 
		\[\begin{split}
			\left(- \ell^{-2-2s} \frac{\sqrt{\pi}}{1+s} \frac{\Gamma(\frac12 +s)}{ \Gamma(1+s)} s \left( \frac{\Im{\tau}}\ell \right)^{-1-2s} \right) \left( \int_\R \varphi_s(w,z) dw \right) &\frac{A_1 \ell^{2+2s}}{N_1 \phi(\ell_2 A_1)} \frac{\zeta(2s+1)}{\zeta(2s+2)} \\ \times  &\sum_{g|N/\gamma} \mu(g) \frac{\phi(g \gamma \ell)}{(g \gamma \ell)^{2+2s}} \prod_{p| g \gamma \ell} \frac{1}{1 - p^{-2 - 2s}}.
		\end{split}\]
		The lemma follows, noting that the simple pole at $s=0$ of $\zeta(2s+1)$ gets canceled by the factor of $s$ from the first term.
	\end{proof}
	
	In the above proof, we used a bound for the generalized Kloosterman sum $K_{i\infty, \rho}(m,n,c)$. Let us prove this bound in the following lemma.
	
	\begin{lemma}\label{lem::weil_bound_general_kloostermansum}
		For any $(m,n)\ne (0,0)$ and for any $\epsilon>0$ \[|K_{i \infty, \rho}(m,n;c)| \ll_{\ell, N, \epsilon} c^{1+\epsilon} \sqrt{|n|}.\]
	\end{lemma}
	\begin{proof}
		Since $K_{i\infty, \rho}(m,n;c)$ is zero if $\gcd(N, c) \ne \gamma$, we may suppose that $\gcd(N, c) = \gamma$. Note that 
		\begin{equation}\label{eqn::kloosterman_sum_explicit_sum}
			K_{i\infty, \rho}(m,n;c) = \sum_{\begin{smallmatrix} d \modc{\ell c} \\ d\equiv (c/\gamma) \delta \modc{N/\gamma} \\ \gcd(d, c) = 1 \end{smallmatrix}} e\left( \frac{md}{\ell c} + \frac{n[d]_c}c \right)
		\end{equation}
		where for co-prime integers $x,m$ we define $[x]_m$ as any inverse of $x$ modulo $m$.
		
		Recall the definitions of $\ell_1$, $\ell_2$, and $N_1$ given above Lemma \ref{lem::analytic_cont_Sigma2_arb_s}. Notice that we have $\gcd(N/\gamma, \gamma) = \gcd(N_1, \gamma)$, and so $N_1 = \ell_1 \gcd(N_1, \gamma) | \ell_1 \gamma$. Thus $\gcd(N_1, \ell_2) = 1$ and $N_1|\ell_1 c$. Also, from $\ell|N$ and $\gcd(N, c) = \gamma$ it follows that $\gcd(c, \ell_2) = 1$. We can now use the Chinese remainder theorem to rewrite $d$ in equation (\ref{eqn::kloosterman_sum_explicit_sum}) as \[d \equiv \ell_2 [\ell_2]_{\ell_1 c} d' + \ell_1 c[\ell_1 c]_{\ell_2} d'' \modc{\ell c}\]
		where $d'$ ranges modulo $\ell_1 c$ and $d''$ ranges modulo $\ell_2$. The condition $\gcd(d,c)=1$ in equation (\ref{eqn::kloosterman_sum_explicit_sum}) is equivalent to $\gcd(d',c)=1$, while the other condition on $d$ is equivalent to \[d' \equiv \left( \tfrac{c}{\gamma} \right) \delta \modc{N_1} \quad \text{ and } \quad d'' \equiv \left( \tfrac{c}{\gamma} \right) \delta \modc{\ell_2}\]
		by the Chinese remainder theorem. Equation (\ref{eqn::kloosterman_sum_explicit_sum}) thus becomes 
		\begin{align*}
			K_{i\infty, \rho}(m,n;c) &= \sum_{\begin{smallmatrix} d' \modc{\ell_1 c}, \; \gcd(d', c)=1 \\ d' \equiv (c/\gamma) \delta \modc{N_1} \end{smallmatrix}} \sum_{\begin{smallmatrix}
			d'' \modc{\ell_2} \\ d'' \equiv (c/\gamma) \delta \modc{\ell_2} \end{smallmatrix} } e\left( \frac{m[\ell_2]_{\ell_1 c} d'}{\ell_1 c} + \frac{m[\ell_1 c]_{\ell_2} d''}{\ell_2} + \frac{n[d']_c}{c} \right)\\
			&= \frac{1}{N_1} e\left(\frac{m[\ell_1 c]_{\ell_2} (\frac c \gamma) \delta}{\ell_2} \right)\sum_{\begin{smallmatrix} d' \modc{\ell_1 c}\\ \gcd(d', c)=1 \end{smallmatrix}} \sum_{r \modc{N_1}} e\left( \frac{m[\ell_2]_{\ell_1 c} d'}{\ell_1 c} + \frac{n[d']_c}{c} + r \cdot \frac{d' - (\frac c\gamma) \delta}{N_1} \right)
		\end{align*}
		where the last equality follows from the fact that the sum of all $k$'th powers of $N_1$'th roots of unity is 0 if $N_1\nmid k$ and is $N_1$ otherwise. 
		Simplifying, we have 
		\begin{equation}\label{eqn::Kloosterman_sum_double_sum}
			K_{i\infty, \rho}(m,n;c) = \frac{1}{N_1} e\left(\frac{m[\ell_1 c]_{\ell_2} (\frac c \gamma) \delta}{\ell_2} \right)\sum_{r \modc{N_1}} e\left(-\frac{r (\frac c\gamma) \delta}{N_1} \right) \sum_{\begin{smallmatrix} d' \modc{\ell_1 c}\\ \gcd(d', c)=1 \end{smallmatrix}}  e\left( \frac{m'_r d'}{\ell_1 c} + \frac{n[d']_c}{c} \right)
		\end{equation}
		where $m'_r = m[\ell_2]_{\ell_1 c} + \frac{1}{N_1} \ell_1 c r$. 
		
		To handle the inner sum over $d'$, write $c = c_1c_2$ where $\mbb P(c_1) = \mbb P(c) \cap \mbb P(\ell_1)$, and where $\gcd(c_2, \ell_1) = 1 = \gcd(c_1, c_2)$. Using the Chinese remainder theorem to rewrite the sum over $d'$ as before, we get 
		\begin{align*}
			\sum_{\begin{smallmatrix} d' \modc{\ell_1 c}\\ \gcd(d', c)=1 \end{smallmatrix}}  e\left( \frac{m'_r d'}{\ell_1 c} + \frac{n[d']_c}{c} \right) &= \sum_{\begin{smallmatrix} d_1 \modc{\ell_1 c_1}, \; d_2 \modc{c_2}\\ \gcd(d_1, c_1)=1 = \gcd(d_2, c_2) \end{smallmatrix}} e\Bigg(\frac{m'_r[\ell_1 c_1]_{c_2} d_2}{c_2} + \frac{m'_r [c_2]_{\ell_1 c_1}d_1}{\ell_1 c_1} + \frac{n \ell_1 [\ell_1 c_1]_{c_2} [d_2]_{c_2}}{c_2} \\
			& \qquad \qquad \qquad \qquad \qquad \qquad \qquad + \frac{n [c_2]_{\ell_1 c_1} [d_1]_{c_1}}{c_1}\Bigg) \\
			&= K(m'_r[\ell_1 c_1]_{c_2}, n[\ell_1c_1]_{c_2}; c_2) \sum_{\begin{smallmatrix}d_1 \modc{\ell_1 c_1} \\ \gcd(d_1, c_1)=1 \end{smallmatrix} } e\left(\frac{m'_r [c_2]_{\ell_1 c_1}d_1}{\ell_1 c_1} + \frac{n [c_2]_{\ell_1 c_1} [d_1]_{c_1}}{c_1} \right).
		\end{align*}
		However, from $\gcd(c, N) = \gamma$ and $\mbb P(\ell_1) = \mbb P(\ell) \cap \mbb P(\gamma)$ we get that $\mbb P(c_1) = \mbb P(\ell_1)$. Thus $\gcd(d_1, c_1) = 1$ if and only if $\gcd(d_1, \ell_1 c_1) = 1$, and so by choosing $[d_1]_{c_1}$ appropriately, the second sum also becomes a Kloosterman sum. Thus
		\begin{equation}\label{eqn::gen_kloosterman_sum_inner_sum_d'}
			\sum_{\begin{smallmatrix} d' \modc{\ell_1 c}\\ \gcd(d', c)=1 \end{smallmatrix}}  e\left( \frac{m'_r d'}{\ell_1 c} + \frac{n[d']_c}{c} \right) = K(m'_r[\ell_1 c_1]_{c_2}, n[\ell_1c_1]_{c_2}; c_2) K(m'_r[c_2]_{\ell_1 c_1}, n[c_2]_{\ell_1c_1} \ell_1; \ell_1c_1).
		\end{equation}
		Using the Weil bound for Kloosterman sums then yields 
		\begin{align*}
			\left|\sum_{\begin{smallmatrix} d' \modc{\ell_1 c}\\ \gcd(d', c)=1 \end{smallmatrix}}  e\left( \frac{m'_r d'}{\ell_1 c} + \frac{n[d']_c}{c} \right) \right| &\le c_2 \tau(c_2) \sqrt{\gcd(m'_r[\ell_1 c_1]_{c_2}, n[\ell_1c_1]_{c_2}, c_2)} \ell_1 c_1 \tau(\ell_1 c_1) \sqrt{\gcd(m'_r[c_2]_{\ell_1 c_1}, n[c_2]_{\ell_1c_1} \ell_1, \ell_1c_1)} \\
			&= \ell_1 c \tau(\ell_1 c) \sqrt{\gcd(m'_r, n, \ell_1 c)} \le \ell_1 c \tau(\ell_1 c) \sqrt{|n|}.
		\end{align*}
		Using this bound in equation (\ref{eqn::Kloosterman_sum_double_sum}), and then using the bound $\tau(c) \ll_{\epsilon} c^{\epsilon}$ for any $\epsilon>0$, the required bound follows.
	\end{proof}
	
	We also required a general formula for $K_{i\infty, \rho}(0,0;c) = \#\{\smalltwosqmat abcd \in \Gamma_\infty \bs \Gamma_0(N) L / \Gamma_\infty ^\ell \}$, proven below. Recall the definitions of $\ell_1$, $\ell_2$, $N_1$, $A_1$, and $A_2$ given above Lemma \ref{lem::analytic_cont_Sigma2_arb_s}.
	
	\begin{lemma}\label{lem::formula_for_K_infty,rho_0,0,c}
		For any $c\in \N$, $\displaystyle K_{i\infty, \rho}(0,0,c) = \frac{A_2}{N_1} \phi\left(A_1 \frac{c}{\gamma} \right)$.
	\end{lemma}
	\begin{proof}
		By equations (\ref{eqn::Kloosterman_sum_double_sum}) and (\ref{eqn::gen_kloosterman_sum_inner_sum_d'}), we have \[K_{i\infty, \rho}(0,0;c) = \frac1{N_1} \sum_{r \modc{N_1}} e\left(-\frac{rc\delta}{\gamma N_1} \right)K\left(\frac{\ell_1 c}{N}r[\ell_1c_1]_{c_2}, 0; c_2 \right) K\left(\frac{\ell_1 c}{N}r[c_2]_{\ell_1c_1}, 0; \ell_1c_1 \right)\]
		with $c = c_1c_2$ where $\gcd(c, \ell_1) = 1 = \gcd(c_1, c_2)$, and $\mbb P(c_1) = \mbb P(c) \cap \mbb P(\ell_1)$. 
		Since $K(m,0;c)$ is Ramanujan's sum of $m$'th powers of primitive $c$'th roots of unity, we use Kluyver's formula (c.f. \cite{kluyver_ramujansumformula}) to get \[K_{i\infty, \rho}(0,0;c) = \frac1{N_1} \sum_{r \modc{N_1}} e\left(-\frac{rc\delta}{\gamma N_1} \right) \sum_{u|c_2, u|\frac{\ell_1cr}{N_1} } \mu\left(\frac{c_2}u \right) u \sum_{v|\ell_1 c_1, v| \frac{\ell_1 c r}{N_1} }\mu \left( \frac{\ell_1 c_1}{v} \right) v.\]
		As $c_2$ and $\ell_1 c_1$ are co-prime, we have 
		\[K_{i\infty, \rho}(0,0;c) = \frac1{N_1} \sum_{r \modc{N_1}} e\left(-\frac{rc\delta}{\gamma N_1} \right) \sum_{u|c_2, v|\ell_1 c_1,  uv|\frac{\ell_1cr}{N_1} } \mu \left( \frac{\ell_1 c}{uv} \right) uv.\]
		Writing $w=uv$, simplifying the inner sum, and then swapping the two sums gives 
		\begin{align*}
			K_{i\infty, \rho}(0,0;c) &= \frac1{N_1} \sum_{r \modc{N_1}} e\left(-\frac{rc\delta}{\gamma N_1} \right) \sum_{w|\frac{\ell_1c}{N_1} \gcd(r, N_1)} \mu \left( \frac{\ell_1 c}{w} \right)w\\
			&= \frac{1}{N_1} \sum_{g|N_1} \sum_{w|\frac{\ell_1 c}{N_1/g} } \mu \left( \frac{\ell_1 c}{w} \right) w \sum_{r' \in (\Z/(N_1/g)\Z)^*} e\left(-\frac{r'(c/\gamma) \delta}{N_1/g} \right).
		\end{align*}
		Since $c/\gamma$ and $N_1$ are co-prime, we have 
		\begin{align*}
			K_{i\infty, \rho}(0,0;c) &= \frac{1}{N_1} \sum_{g|N_1} \sum_{w|\frac{\ell_1 c}{g} } \mu \left( \frac{\ell_1 c}{w} \right) w \sum_{r' \in (\Z/g\Z)^*} e\left(-\frac{r'\delta}{g} \right)\\
			&= \frac{1}{N_1} \sum_{g|N_1} \sum_{w|\frac{\ell_1 c}g} \sum_{a|\delta, a|g} \mu(gw) \frac{\ell_1 c}{gw} \mu\left( \frac ga\right) a,
		\end{align*}
		by Kluyver's formula again. Setting $b=wg$ gives 
		\begin{align*}
			K_{i\infty, \rho}(0,0;c) &= \frac{\ell_1c}{N_1} \sum_{b|\ell_1 c} \;\sum_{g|N_1, g|b} \;\sum_{a|\delta, a|g} \frac{\mu(b)}{b} \mu\left( \frac ga\right) a \\
			&= \frac{\ell_1c}{N_1} \sum_{b|\ell_1 c} \frac{\mu(b)}{b} \sum_{a|\gcd(\delta, b, N_1)} a\sum_{g'|\gcd(b,N_1)/a}   \mu\left(g'\right) \\
			&= \frac{\ell_1c}{N_1} \sum_{b|\ell_1 c} \frac{\mu(b)}{b} \sum_{a|\gcd(\delta, b, N_1)} a \cdot \begin{cases}
				1 & \text{if } \gcd(b,N_1) = a,\\
				0 & \text{otherwise,}
			\end{cases}\\
			&= \frac{\ell_1c}{N_1} \sum_{b|\ell_1 c, \;\gcd(b, N_1)|\delta } \frac{\mu(b)}{b} \gcd(b, N_1).
		\end{align*}
		Now, note that $\ell_1 \ell_2 = \ell = \frac{N/\gamma}{\gcd( N/\gamma, \gamma)} = \frac{\ell_2 N_1}{\gcd(N_1, \gamma)}$, and so $N_1 = \ell_1 \gcd(N_1, \gamma)$. Suppose $\delta$ and $N_1$ share a common prime factor $p$; as $\gamma$ and $\delta$ are co-prime and as $\gcd(N_1, \gamma)|\gamma$, it follows that $p|\ell_1$. However $\mbb P(\ell_1) \subseteq \mbb P(\gamma)$ by construction and so $p|\gamma$, a contradiction. Thus $\gcd(\delta, N_1)=1$ and so \[K_{i\infty, \rho}(0,0;c) = \frac{\ell_1c}{N_1} \sum_{b|\ell_1 c, \;\gcd(b, N_1)=1 } \frac{\mu(b)}{b}.\]
		Writing $\ell_1 \gamma = A_1A_2$ with $\gcd(A_2, N_1) = 1 = \gcd(A_1, A_2)$ and $\mbb P(A_1) \subseteq \mbb P(N_1)$, and using the fact that $\sum_{d|n} \frac{\mu(d)}d = \frac{\phi(n)}{n}$ yields the required result. 
	\end{proof}
	
	Having now established the analytic continuation for any $\Re s> -\frac14$, we explicitly evaluate at $s=0$ to get the following expansion.
	\begin{lemma}
		The analytic continuation at $s=0$ of $\Sigma^{(2)}_{N, \rho}(s)$ is explicitly given by \[\frac{-3}{[SL_2(\Z): \Gamma_0(N)]\Im{\tau}} - \frac{4\pi^3}{\ell^2} \sum_{m\ge 1} m \left(a_\rho(m, 0) + \sum_{n\ge 1} \Big(a_\rho(m,-n)e(nz) + a_\rho(m,n)e(-n\bar z) \Big) \right) e^{2\pi i m \tau / \ell}\]
		where \[a_\rho(m,n) := \sum_{c\ge 1} \frac{K_{i\infty, \rho}(m,n;c) }{c^{2}}.\]
	\end{lemma}
	\begin{proof}
		In the second term given in Lemma \ref{lem::analytic_cont_Sigma2_arb_s}, setting $s=0$ gives
		\[\sum_{(m,n)\in \Z^2\bs\{(0,0)\}} \Im z \left(\frac{1}{\ell^{2}}\int_\R \frac{e^{-2\pi i m t}}{(\frac \tau \ell + t)^{2}} dt \right) \left(\int_\R \frac{e^{-2\pi i n w}}{(w- z)(w- \bar z)} dw\right) \sum_{c\ge 1} \frac{K_{i\infty, \rho}(m,n;c) }{c^{2}}\]
		which can be written as \[\frac y{\ell^2} \sum_{(m,n)\in \Z^2\bs\{(0,0)\}} g_m\left(\frac \tau \ell, \frac \tau \ell \right) g_n(-z, -\bar z) a_\rho(m,n).\]
		Lemma \ref{lem::compute_g_n} then yields \[- \frac{4\pi^3}{\ell^2} \sum_{m\ge 1} m \left(a_\rho(m, 0) + \sum_{n\ge 1} \Big(a_\rho(m,-n)e(nz) + a_\rho(m,n)e(-n\bar z) \Big) \right) e^{2\pi i m \tau / \ell}.\]
		Finally, using the fact that $\Gamma(\frac12) = \sqrt \pi$, $\zeta(2) = \frac{\pi^2}6$, $\int_\R \varphi_0(w,z)dw = \Im z g_0(-z, -\bar z) = \pi$, and that $s\zeta(2s+1)$ converges to $\frac12$ as $s\to 0$, the first term is \[\frac{3\ell}{\Im{\tau}} \cdot \frac{A_1}{N_1 \phi(\ell_2 A_1)} \sum_{g|N/\gamma} \mu(g) \frac{\phi(g \gamma \ell)}{(g \gamma \ell)^2} \prod_{p| g \gamma \ell} \frac{1}{1 - p^{-2}}.\]
		
		Since \(\displaystyle \phi(g \gamma \ell) = g\gamma \ell \prod_{p|g \gamma \ell} \left(1 - p^{-1} \right)\),
		the first term becomes \[\frac{3}{\Im{\tau}} \cdot \frac{A_1}{N_1 \phi(\ell_2 A_1)} \sum_{g|N/\gamma}  \frac{\mu(g)}{g \gamma} \prod_{p| g \gamma \ell} \frac{1}{1 + p^{-1}} = \frac{3}{\Im{\tau} N} \cdot \frac{\ell_2A_1}{\phi(\ell_2 A_1)} \sum_{g|N_1 \ell_2}  \frac{\mu(g)}{g} \prod_{p|g\gamma \ell_1 \ell_2}\frac1{1+p^{-1}},\]
		where use the fact that $N = \gamma N_1 \ell_2$. 
		As $\gcd(N_1, \ell_2) = 1$, we can then write the sum over $g$ as \[\sum_{g_1|N_1, g_2|\ell_2}  \frac{\mu(g_1) \mu(g_2)}{g_1g_2} \prod_{p|g_1g_2A_1A_2 \ell_2} \frac1{1+p^{-1}}.\] 
		From the fact that $\gcd(A_2, N_1) = 1 = \gcd(A_2, A_1)$ and that $\gcd(\ell_2, A_1 A_2) = 1$ (as $\ell_2$ is co-prime to $\ell_1$ and $\gamma$), the above sum is \[\left( \prod_{p|A_2 } \frac1{1+p^{-1}}\right) \left( \sum_{g_1|N_1} \frac{\mu(g_1) }{g_1} \prod_{p|g_1A_1} \frac1{1+p^{-1}} \right) \left(\sum_{g_2|\ell_2} \frac{\mu(g_2)}{g_2} \prod_{p|g_2\ell_2} \frac1{1+p^{-1}}\right) .\]
		As $g_2|\ell_2$ anyway, we have \[\sum_{g_2|\ell_2} \frac{\mu(g_2)}{g_2} \prod_{p|g_2\ell_2} \frac1{1+p^{-1}} = \sum_{g_2|\ell_2} \frac{\mu(g_2)}{g_2} \prod_{p|\ell_2} \frac1{1+p^{-1}} = \frac{\phi(\ell_2)}{\ell_2} \prod_{p|\ell_2} \frac{1}{1+p^{-1}}.\]
		For the sum over $g_1$, using the fact that $\mu(n) = 0$ if $n$ is not square-free, we have
		\begin{align*}
			\sum_{g_1|N_1} \frac{\mu(g_1) }{g_1} \prod_{p|g_1A_1} \frac1{1+p^{-1}} &= \sum_{S \subseteq \mbb P(N_1)} \frac{(-1)^{|S|} }{\prod_{p\in S} p} \prod_{p\in S \cup \mbb P(A_1)} \frac{1}{1+p^{-1}} \\
			&= \left(\sum_{S_1 \subseteq \mbb P(N_1) \bs \mbb P(A_1)} (-1)^{|S_1|}  \prod_{p\in S_1}\frac{1}{p+1} \right) \left( \sum_{S_2 \subseteq \mbb P(A_1)} \frac{(-1)^{|S_2|}}{\prod_{p \in S_2} p} \prod_{p\in \mbb P(A_1)} \frac{1}{1+p^{-1}} \right) \\
			&= \left( \frac{\phi(A_1)}{A_1} \prod_{p|A_1} \frac{1}{1+p^{-1}} \right) \prod_{p|N_1, p\nmid A_1} \left(1 - \frac{1}{p+1} \right) = \frac{\phi(A_1)}{A_1} \prod_{p|N_1} \frac1{1+p^{-1}}.
		\end{align*}
	
		Hence the first term is \[\frac3{\Im \tau N} \cdot \frac{\ell_2 A_1}{\phi(\ell_2 A_1)} \frac{\phi(\ell_2)}{\ell_2}  \frac{\phi(A_1)}{A_1} \left( \prod_{p|A_2} \frac{1}{1+p^{-1}}\right) \left(\prod_{p|\ell_2} \frac{1}{1+p^{-1}} \right)\left(\prod_{p|N_1} \frac{1}{1+p^{-1}} \right) \]
		and since $\mbb P(A_2) = \mbb P(\ell_1 \gamma) \bs \mbb P(N_1)$, $\mbb P(\ell_1)\subseteq \mbb P(\gamma)$, and $N_1 \ell_2 = N/\gamma$, this simplifies to \[ \frac{3}{\Im \tau N} \prod_{p|A_2N/\gamma} \frac{1}{1+p^{-1}} = \frac{3}{\Im \tau N} \prod_{p|N} \frac{1}{1+p^{-1}}\]
		as required.
	\end{proof}
	
	\subsection{Analytically continuing $\Sigma_{N, \rho}^{(1)}(s)$ to $s=0$}
	The sum $\Sigma_{N, \rho}^{(1)}(s)$ converges absolutely and uniformly for $s$ in rectangles of the form $-\frac12 < \sigma_0 \le \Re{s} \le \sigma_1$ and $|\Im{s}| \le R$. This follows exactly as in Lemma 3.3 of \cite{construction_Maass_form}. We now need to evaluate $\Sigma_{N, \rho}^{(1)}(s)$ at $s=0$. 
	\begin{lemma}
		If $\Im{\tau} > \max\{y, 1/y\}$,
		\[\begin{split}
			\Sigma^{(1)}_{N, \rho}(0) ={}& \frac{4\pi ^3}{\ell^2} \sum_{m\ge 1} m \sum_{n\ge 1}\left( a_\rho (m, n) e(-n \bar z) + a_\rho (m, -n) e(nz)\right) e^{2\pi i m \tau / \ell} \\
			&\quad + \frac{2\pi^2}{\ell^{3/2}} \sum_{m\ge 1} \sqrt m \left(\sum_{n,c \ge 1} \frac{e^{-2\pi i n \bar z} K_{i \infty, \rho} (m, n; c)}{c \sqrt n} J_1 \left( \frac{4\pi}{c} \sqrt{\frac{mn}{\ell}} \right) \right) e^{2\pi i m \tau / \ell}\\
			&\quad + \frac{2\pi^2}{\ell^{3/2}} \sum_{m\ge 1} \sqrt m \left(\sum_{n,c \ge 1} \frac{e^{2\pi i n z} K_{i \infty, \rho} (m, -n; c)}{c \sqrt n} I_1 \left( \frac{4\pi}{c} \sqrt{\frac{mn}{\ell}} \right) \right) e^{2\pi i m \tau / \ell}
		\end{split}\]
		where \[a_\rho(m,n) = \sum_{c\ge 1} \frac{K_{i\infty, \rho}(m,n;c)}{c^2}.\]
	\end{lemma}
	\begin{proof}
		Using Poisson summation twice,
		\begin{align*}
			\Sigma_{N, \rho}^{(1)}(0) &= \sum_{c\ge 1} \sum_{\smalltwosqmat abcd \in \Gamma_\infty \bs \Gamma_0(N) L / \Gamma_\infty ^\ell } \sum_{m\in \Z} \sum_{n\in \Z} \frac{1}{j(\smalltwosqmat{1}{n}{0}{1} \smalltwosqmat{a} bcd \smalltwosqmat 1{\ell m}01, \tau)^2}\left( \varphi_0\left(n + \smalltwosqmat abcd (\tau + \ell m), z \right) - \varphi_0\left(n + \tfrac ac, z \right) \right) \\
			&= y\sum_{c\ge 1} \sum_{\smalltwosqmat abcd \in \Gamma_\infty \bs \Gamma_0(N) L / \Gamma_\infty ^\ell } \sum_{m\in \Z} \sum_{n\in \Z} \frac{1}{ c^2 (\tau + \frac dc + \ell m)^2} \Biggl( \frac{1}{\left(n + \frac ac - \frac{1}{c^2 \ell (\frac \tau\ell + m + \frac{d}{ c\ell})} - z\right) \left(n + \frac ac - \frac{1}{c^2 \ell (\frac \tau\ell + m + \frac{d}{ c\ell})} - \bar z\right)} \\
			& \qquad \qquad - \frac1{(\frac ac + n - z) (\frac ac + n - \bar z)} \Biggl) \\
			&= y\sum_{c\ge 1} \sum_{\smalltwosqmat abcd \in \Gamma_\infty \bs \Gamma_0(N) L / \Gamma_\infty ^\ell } \sum_{m,n\in \Z} \Biggl( \int_\R \frac{e^{-2\pi i m t}}{ c^2 (\tau + \frac dc + \ell t)^2} \int_\R \frac{e^{-2\pi i n w} dw dt}{\left(w + \frac ac - \frac{1}{c^2 \ell (\frac \tau\ell + t + \frac{d}{ c\ell})} - z\right) \left(w + \frac ac - \frac{1}{c^2 \ell (\frac \tau\ell + t + \frac{d}{ c\ell})} - \bar z\right)} \\
			& \qquad \qquad - \int_\R \frac{e^{-2\pi i m t}dt}{ c^2 (\tau + \frac dc + \ell t)^2} \int_\R\frac {e^{-2\pi i n w} dw}{(\frac ac + w - z) (\frac ac + w - \bar z)} \Biggl). \\
		\end{align*}
		The second term is just \[\frac1{c^2 \ell^2} g_m\left(\frac \tau \ell + \frac{d}{c\ell}, \frac \tau \ell + \frac{d}{c\ell} \right) g_n\left(\frac ac - z, \frac ac - \bar z \right) = \begin{cases}
			0 & \text{if } m \le 0, \\
			-\frac{4\pi ^2 m e(md / \ell c)}{c^2 \ell^2} e^{2\pi i m \tau / \ell} \cdot \frac{\pi}{y} e(na/c) e(-nx)e^{-2\pi |n| y} & \text{if } m > 0,
		\end{cases}\]
		using Lemma \ref{lem::compute_g_n}. The first double integral can be rewritten using the change of variables $t\mapsto t - \frac d{c\ell} $, $w\mapsto w - \frac ac$, to get \[ \frac{e(md/c\ell)}{c^2 \ell^2} \int_\R \frac{e^{-2\pi i m t} }{(t + \frac \tau \ell)^2} \int_\R \frac{e(na/c) e^{-2\pi i n w}} {\left(w - \frac{1}{c^2 \ell (\frac \tau\ell + t)} - z\right) \left(w - \frac{1}{c^2 \ell (\frac \tau\ell + t)} - \bar z\right) }dwdt.\]
		Let $\tau = u+iv$; then using $t\mapsto t - \frac \tau \ell$, the double integral becomes \[\frac{e\left( \frac{md}{c\ell} + \frac{na} c\right) e^{2\pi i m \tau / \ell}}{c^2 \ell^2} \int_{\R + i \frac v\ell} \frac{e^{-2\pi i m t} }{t^2} \int_\R \frac{e^{-2\pi i n w}} {\left(w - \frac{1}{c^2 \ell t} - z\right) \left(w - \frac{1}{c^2 \ell t} - \bar z\right) }dwdt.\]
		The inner integral is just $g_n(-z - \frac{1}{c^2 \ell t}, -\bar z - \frac1{c^2 \ell t})$, and using the conditions on $v$ we have \[g_n \left( -z - \frac{1}{c^2 \ell t}, -\bar z - \frac1{c^2 \ell t} \right) = \frac \pi y e^{-2\pi i n / c^2 \ell t} e(-nx) e^{-2\pi |n| y}.\]
		Thus the double integral reduces to \[\frac{e\left( \frac{md}{c\ell} + \frac{na} c\right) e^{2\pi i m \tau / \ell}}{c^2 \ell^2} \cdot \frac \pi y \cdot e(-nx) e^{-2\pi |n| y}\int_{\R + i \frac v\ell} \frac{\exp(-2\pi i (m t + \frac{n}{c^2 \ell t})) }{t^2} dt.\]
		Let \[I_{m,n,c}(\xi) = \int_{\R + i \xi} \frac{\exp(-2\pi i (m t + \frac{n}{c^2 \ell t})) }{t^2} dt.\]
		\begin{itemize}
			\item For $n\in \Z$ and $m\le 0$, notice that for $\Im{t} = \xi$ we get \[\left|\exp\left( -2\pi i \left(m t + \frac{n}{c^2 \ell t} \right)\right) \right| = \exp\left( 2\pi \left( m\xi - \frac{n \xi}{c^2 \ell |t|^2}\right) \right) \le \exp\left( 2\pi \left( m\xi +\frac{|n|}{c^2 \ell \xi}\right) \right)\]
			which for $m<0$ goes to 0 as $\xi \to \infty$. For $m=0$, we have \[|I_{m,n,c}(\xi)| \le \exp\left(\frac{2\pi |n|}{c^2 \ell \xi} \right)\int_{\R+i\xi} \frac{dt}{t^2} \to 0\]
			as $\xi \to \infty$. Thus			
			\[\lim_{\xi \to \infty} I_{m,n,c}(\xi) = 0\]
			for $m\le 0$. Since the integrand is analytic for $\Re{t} > 0$, by Cauchy's Theorem we have $I_{m,n,c} (v/\ell) = I_{m,n,c}(\xi)$ for all $\xi > 0$, and so $I_{m,n,c} (v / \ell)=0$. 
			
			\item For $n = 0$ and $m\in \N$, we have by Lemma \ref{lem::compute_g_n}
			\begin{align*}
				I_{m,0,c}(v/\ell) &= \int_\R \frac{e^{-2 \pi i m t} e^{2\pi m v/ \ell}}{(t + iv/\ell)^2} dt = e^{2\pi m v/ \ell} g_m(i v/ \ell, i v / \ell) = - 4\pi ^ 2 m.
			\end{align*}
			
			\item For $n\ne 0$ and $m\in \N$, using the change of variables $t \mapsto i \ell^{-1/2} |n|^{1/2} m^{-1/2} c^{-1} t$ and then using Cauchy's Theorem to shift the line of integration to $\xi + i\R$ for $\xi>0$, we have 
			\begin{align*}
				I_{m,n,c}(v/\ell) &= \int_{cv \sqrt{m / \ell|n|} + i\R} \frac{\exp \left(\frac{2\pi }{c} \sqrt{ \frac{m|n|}{\ell}} (t - \mr{sgn}(n) / t)\right)}{c^{-1} i \sqrt{|n|/m\ell} \cdot t^2} dt \\
				&= \frac{c\sqrt{\ell m/|n|}}{i} \int_{\xi + i\R} \frac{\exp \left(\frac{2\pi }{c} \sqrt{ \frac{m|n|}{\ell}} (t - \mr{sgn}(n) / t)\right)}{t^2} dt \\
				&= \frac{c\sqrt{\ell m/|n|}}{i}\cdot 2\pi i \mc{L}^{-1} \left( \frac{\exp \left(- \mr{sgn}(n) \frac{2\pi }{c} \sqrt{ \frac{m|n|}{\ell}}\cdot \frac1t\right)}{t^2}\right)\left( \frac{2\pi }{c} \sqrt{ \frac{m|n|}{\ell}} \right) \\
				&= 2\pi c \sqrt{ \frac{\ell m}{|n|}} \cdot \begin{cases}
					J_1\left(2\cdot \frac{2\pi }{c} \sqrt{ \frac{m|n|}{\ell}} \right) & \text{if } n>0,\\
					I_1\left(2\cdot \frac{2\pi }{c} \sqrt{ \frac{m|n|}{\ell}} \right) & \text{if } n<0,
				\end{cases}
			\end{align*}
			where we use the Mellin integral for the inverse Laplace transform of $t^{-2} e^{-\mr{sgn}(n)\kappa / t}$ \cite[c.f.][{}29.3.80-81]{handbook_fns}.
		\end{itemize} 
		Therefore 
		\begin{align*}
			\Sigma_{N, \rho}^{(1)}(0) 
			&= \frac{\pi}{\ell^2}\sum_{c\ge 1} \sum_{\smalltwosqmat abcd \in \Gamma_\infty \bs \Gamma_0(N) L / \Gamma_\infty ^\ell } \sum_{n\in \Z} \sum_{m\ge 1} \frac{1}{c^2} e\left( \frac{md}{c\ell} + \frac{na} c\right)\left(I_{m,n,c}(v/\ell) + 4\pi ^2 m \right) e(-nx) e^{-2\pi |n| y} e^{2\pi i m \tau / \ell} \\
			&= \frac{\pi}{\ell^2}\sum_{m\ge 1} \Bigg(\sum_{n\ge 1} \sum_{c\ge 1} \frac{K_{i\infty, \rho}(m,n;c)}{c^2}\left( 2\pi c \sqrt{ \frac{\ell m}{n}} J_1\left( \frac{4\pi }{c} \sqrt{ \frac{mn}{\ell}}\right) + 4\pi ^2 m \right) e^{-2\pi in \bar z} \\
			&\qquad \qquad \quad + \sum_{n\ge 1} \sum_{c\ge 1} \frac{K_{i\infty, \rho}(m,-n;c)}{c^2}\left( 2\pi c \sqrt{ \frac{\ell m}{n}} I_1\left( \frac{4\pi }{c} \sqrt{  \frac{mn}{\ell}}\right) + 4\pi ^2 m \right) e^{2\pi i n z}  \Bigg)e^{2\pi i m \tau / \ell}
		\end{align*} 
		which yields the desired result.
	\end{proof}
	\section{Some Useful Lemmas, and Proofs of Propositions \ref{prop::mellin_transform_z_converges_infty_only} and \ref{prop::mellin_transform_z_converges}}\label{sctn::useful_lemmas}
	In this section, we first prove a functional equation relating $H_{N,-1/Nz}^*$ and $H_{N,z}^*$. Then, we analyze certain integrals of the various pieces of the Fourier expansion, and in doing so we deduce Propositions \ref{prop::mellin_transform_z_converges_infty_only} and \ref{prop::mellin_transform_z_converges} directly. 
	
	\subsection{A Functional Equation}
	Let us first relate $H_{N, -1/Nz}^*$ with $H_{N,z}^*$.
	\begin{lemma}\label{lem::func_eqn_H_N,z_and_H_N,WNz}
		Let $W_N = \smalltwosqmat{0}{-1}{N}{0}$ be the Fricke involution. For all $z\in \mbb H$, we have $ H_{N, W_Nz}^* =  H_{N,z}^*|_2[W_N]$.
	\end{lemma}
	\begin{proof}
		We first prove that \[\frac{N^{1+s}}{(N\tau)^2 |N\tau|^{2s}} P_{N,s}(W_N\tau, W_Nz) = P_{N,s}(\tau, z)\]
		for all $\Re s>0$. Note the elementary identities $j(W_N, \mf z) = N\mf z$, \[W_N \mf z - W_Nz = \frac{N(\mf z - z)}{j(W_N, \mf z)j(W_N, z)}, \quad \text{ and } \quad j(M,W_N\mf z) j(W_N, \mf z) = j(MW_N, \mf z)\]
		where $\mf z\in \mbb H$ and $M\in \Gamma_0(N)$ are arbitrary. These identities then imply that
		\[\varphi_s(W_N \mf z, W_N z) = \frac{\Im{z}^{1+s} N^{-1-s} |z|^{-2-2s}}{\frac{N(\mf z - z)}{j(W_N, \mf z)j(W_N, z)} \cdot \frac{N(\mf z - \bar z)}{j(W_N, \mf z)j(W_N, \bar z)} \cdot \left| \frac{N(\mf z - \bar z)}{j(W_N, \mf z)j(W_N, \bar z)} \right|^{2s}} = j(W_N, \mf z)^2 |j(W_N, \mf z)|^{2s} N^{-1-s} \varphi_s(\mf z, z) \]
		for all $\mf z, z\in \Gamma_0(N)$, $\mf z \ne z$. Thus, using the fact that $W_N^{-1} \Gamma_0(N)W_N = \Gamma_0(N)$, we have
		
		\begin{align*}
			\frac{N^{1+s}P_{N,s}(W_N \tau, W_Nz)}{j(W_N, \tau)^2 |j(W_N, \tau)|^{2s}} &= \sum_{M\in \Gamma_0(N)} \frac{N^{1+s}\varphi_s(MW_N \tau, W_N z)}{j(MW_N, \tau)^2 |j(MW_N, \tau)|^{2s}} = \sum_{M\in \Gamma_0(N)} \frac{N^{1+s}\varphi_s(W_NM \tau, W_N z)}{j(W_NM, \tau)^2 |j(W_NM, \tau)|^{2s}} \\
			&= \sum_{M\in \Gamma_0(N)} \frac{j(W_N, M\tau)^2 |j(W_N,  M\tau)|^{2s}\varphi_s(M\tau, z)}{j(W_NM, \tau)^2 |j(W_NM, \tau)|^{2s}} = P_{N,s}(\tau, z).
		\end{align*}
		Replacing $\tau$ with $-1/N\tau$ (and noting that $W_N^2 \tau = \tau$ and $j(W_N, W_N\tau) = -1/\tau$) yields 
		\[P_{N,s}\left(\tau, \frac{-1}{Nz}\right) = \frac{1}{N^{1+s} \tau^2 |\tau|^{2s}}P_{N,s}\left(\frac{-1}{N\tau}, z\right).\]
		The proposition now follows by passing to the analytic continuation at $s=0$.
	\end{proof}
	Taking the limit as $z\to i\infty$ then yields the following corollary.
	
	\begin{corollary}\label{cor::eisenstein_func_eqn}
		For all $\tau\in \mbb H$, we have $ H_{N, 0}^*(\tau) = H_{N,i\infty}^*|_2[W_N](\tau)$. 
	\end{corollary}
	
	\subsection{Proof of Propositions \ref{prop::mellin_transform_z_converges_infty_only} and \ref{prop::mellin_transform_z_converges}}
	Suppose $\rho\in \{0,i\infty\}$; then the cusp widths are $\ell_{i\infty} = 1$ and $\ell_0=N$. For $y := \Im z$ large enough and any $1/y \le t_0 \le y$, we rewrite $L_{N,z\rho}(s)$ by expressing $H_{N,z}(\tau)$ in terms of its Fourier expansions at $i\infty$ and at $0$. By Proposition \ref{prop::fourier_at_arb_cusps}, the Fourier expansion at 0 for $\Im\tau$ large enough is 
	\[\begin{split}
		\tau^{-2}H_{N,z}^*\left(\frac{-1}{\tau} \right) ={}& \frac{3}{\pi [SL_2(\Z): \Gamma_0(N)] \Im{\tau}} + \frac{4\pi^2}{N^2} \sum_{m\ge 1} m \left( \sum_{c\ge 1} \frac{K_{i\infty, 0}(m,0;c) }{c^{2}}\right)  e^{2\pi i m \tau/N } \\
		& \quad - \frac{2\pi}{N^{3/2}} \sum_{m\ge 1} \sqrt m \left(\sum_{n,c \ge 1} \frac{e^{-2\pi i n \bar z} K_{i \infty, 0} (m, n; c)}{c \sqrt n} J_1 \left( \frac{4\pi}{c} \sqrt{\frac{mn}{N}} \right) \right) e^{2\pi i m \tau/N}\\
		&\quad - \frac{2\pi}{N^{3/2}} \sum_{m\ge 1} \sqrt m \left(\sum_{n,c \ge 1} \frac{e^{2\pi i n z} K_{i \infty, 0} (m, -n; c)}{c \sqrt n} I_1 \left( \frac{4\pi}{c} \sqrt{\frac{mn}{N}} \right) \right) e^{2\pi i m \tau/N}.
	\end{split}\]
	Now, the substitution $t\mapsto \frac1t$ for the first integral below and $t\mapsto \frac{1}{Nt}$ for the other two, along with Lemma \ref{lem::func_eqn_H_N,z_and_H_N,WNz} gives
	\begin{align*}
		\int_0^{t_0} t^{s-1}\left(H_{N,z}^*(it)  + \frac{3}{[SL_2(\Z) : \Gamma_0(N)] \pi t}\right)dt &= -\int_{1/t_0}^{\infty} t^{1-s}\left(\frac{1}{(it)^2} H_{N,z}^*\left(\frac{-1}{it} \right)  - \frac{3}{[SL_2(\Z) : \Gamma_0(N)] \pi t}\right)dt,\\
		\int_{t_0}^{\infty} t^{s-1}\left(H_{N,-1/Nz}^*(it)  - \frac{3}{[SL_2(\Z) : \Gamma_0(N)] \pi t}\right)dt &= N^{1-s} \int_{Nt_0}^{\infty} t^{s-1}\left(\frac{1}{(it)^2} H_{N,z}^*\left(\frac{-1}{it} \right)  - \frac{3}{[SL_2(\Z) : \Gamma_0(N)] \pi t}\right)dt ,\\ 
		\int_0^{t_0} t^{s-1}\left(H_{N,-1/Nz}^*(it)  + \frac{3}{[SL_2(\Z) : \Gamma_0(N)] \pi t}\right)dt &= -N^{1-s} \int_{1/Nt_0}^{\infty} t^{1-s}\left(H_{N,z}^*(it)  - \frac{3}{[SL_2(\Z) : \Gamma_0(N)] \pi t}\right)dt.
	\end{align*}
	
	Hence we have the (formal) decomposition
	\begin{align}
		L_{N,z}(s) &= -\frac{6}{[SL_2(\Z):\Gamma_0(N)]\pi} \frac{t_0^{s-1}}{s-1} + \sum_{k=1}^3 a_k(i\infty) I_{i\infty}^{(k)}(s,z;t_0) + I^{(4)}(s,z;t_0) - \sum_{k=1}^3 a_k(0) I_{0}^{(k)}\left(2-s,z;\frac1{t_0} \right),\label{eqn::decompose_L_N,z}\\
		L_{N,-1/Nz}(s) &= -\frac{6}{[SL_2(\Z):\Gamma_0(N)]\pi} \frac{t_0^{s-1}}{s-1} + N^{1-s} \sum_{k=1}^3 a_k(0) I_{0}^{(k)}(s,z;Nt_0) - N^{1-s} I^{(4)}\left(2-s,z;\frac1{Nt_0}\right) \nonumber\\
		&\qquad \qquad \qquad  -N^{1-s}\sum_{k=1}^3 a_k(i\infty) I_{i\infty}^{(k)}\left(2-s,z;\frac1{Nt_0} \right), \label{eqn::decompose_L_N,-1/Nz}
	\end{align}
	where 
	\[I_{\rho}^{(1)}(s,z;t_0) = I_\rho^{(1)}(s;t_0):= \frac1{\ell_\rho} \int_{t_0}^\infty t^{s-1} \sum_{m\ge 1}  j_{N,m}(\rho) e^{-2\pi m t/\ell_\rho} dt, \quad a_1(\rho) = 1\]
	\[I_{\rho}^{(2)}(s,z;t_0) := \int_{t_0}^\infty t^{s-1}\sum_{m,n,c\ge 1} \sqrt{\frac{m}{n}} \frac{K_{i\infty, \rho}(m,-n;c)}{c} I_1\left(\frac{4\pi \sqrt{mn}}{c\sqrt{\ell_\rho}} \right) e^{2\pi i n z} e^{-2\pi m t/\ell_\rho} dt, \quad a_2(\rho)= -\frac{2\pi}{\ell_\rho^{3/2}},\]
	\[I_\rho^{(3)}(s,z;t_0) := \int_{t_0}^\infty t^{s-1} \sum_{m,n,c\ge 1} \sqrt{\frac{m}{n}} \frac{K_{i\infty, \rho}(m,n;c)}{c} J_1\left(\frac{4\pi \sqrt{mn}}{c\sqrt{\ell_\rho} } \right) e^{-2\pi i n \bar z} e^{-2\pi m t/\ell_\rho}dt, \quad \quad a_3(\rho)= -\frac{2\pi}{\ell_\rho^{3/2}},\]
	\[I^{(4)}(s,z;t_0) := \int_{t_0}^\infty t^{s-1} \left(\frac{1}{e^{2\pi (t + iz)}-1} - \frac{1}{e^{2\pi (t + i\bar z)}-1} \right) dt.\]
	
	In order to control the Kloosterman sums, we need to establish the Weil bound for $K_{i\infty, 0}(m,n;c)$, a stronger result than that given in Lemma \ref{lem::weil_bound_general_kloostermansum}.
	
	\begin{lemma}
		If $(m,n)\ne (0,0)$, then $\displaystyle |K_{i\infty, 0}(m, n;c)| \le \tau(c)\sqrt{c} \sqrt{\gcd(m, n, c)}$.
	\end{lemma}
	\begin{proof}
		If $\gcd(c,N)\ne 1$, then the Kloosterman sum is in fact zero, and so the bound trivially holds. Thus assume $N$ and $c$ are co-prime; we then have 
		\begin{align*}
			K_{i\infty, 0}(m,n;c) &= \sum_{\smalltwosqmat abcd \in \Gamma_\infty \bs \Gamma_0(N)S / \Gamma_\infty^N} e\left(\frac1c\left(m(\tfrac dN) + na \right) \right) = \sum_{\begin{smallmatrix} d \modc{Nc}, N|d\\ (c,d)=1, ad\equiv 1\modc{c} \end{smallmatrix}} e\left(\frac1c\left(m(\tfrac dN) + na \right) \right) \\
			&= \sum_{\begin{smallmatrix} d' \modc{c}, (c,d')=1\\ aNd'\equiv 1\modc{c} \end{smallmatrix}} e\left(\frac1c\left(md' + na \right) \right) = \sum_{a\in (\Z/c\Z)^\times} e\left(\frac1c\left(m [N a]_c + na \right) \right) = K(m[N]_c, n;c)
		\end{align*}
		where $[x]_c$ for $x\in (\Z/c\Z)^\times$ denotes the inverse modulo $c$. Note that this is well-defined since the Kloosterman sum $K(m,n;c)$ depends only on $m$ and $n$ modulo $c$. Thus by the Weil bound for the classical Kloosterman sums we have \[|K_{i\infty, 0}(m,n;c)| = |K(m [N]_c,n;c)|\le \tau(c)\sqrt{c} \sqrt{\gcd(m [N]_c, n, c)} = \tau(c)\sqrt{c} \sqrt{\gcd(m, n, c)},\]
		the latter equality holding as $\gcd(c, [N]_c)=1$.
	\end{proof}
	
	We can now establish the convergence properties of the integral $I^{(1)}_\rho(s;t_0)$. Here, $\Gamma(s, a)$ is the (upper) incomplete Gamma function given by \[
	\Gamma(s,a) = \int_a^\infty t^{s-1} e^{-t} dt.\]
	
	\begin{lemma}\label{lem::I1_convergence}
		For any $t_0>0$ and any $s\in \C$, the following function (for fixed $t_0$) is entire
		\[I^{(1)}_\rho(s;t_0) = \ell_\rho^{s-1} \sum_{m\ge 1}j_{N,m}(\rho) \frac{\Gamma(s, 2\pi mt_0/\ell_\rho)}{(2 m\pi)^s}.\]
		Also for fixed $s$, we have\[\lim_{t_0\to \infty} I^{(1)}_\rho(s;t_0) = 0,\]
		and if $\Re s>\frac52$, then \[\lim_{t_0\to 0^+ } I^{(1)}_\rho(s;t_0) = \frac{\ell_\rho^{s-1}\Gamma(s)} {(2\pi)^s}\sum_{m\ge 1} \frac{j_{N,m}(\rho)}{m^s}\]
	\end{lemma}
	\begin{proof}
		First, we use the substitution $t\mapsto \ell_\rho t$ to get 
		\begin{equation}\label{eqn::I1_as_regularized_mellin}
			I^{(1)}_{\rho}(s;t_0) = \ell_\rho^{s-1} \int_{t_0/\ell_\rho }^\infty t^{s-1} \sum_{m\ge 1}  j_{N,m}(\rho) e^{-2\pi m t} dt = \ell_\rho^{s-1} \int_{t_0/\ell_\rho}^\infty t^{s-1} \left(H_{N,\rho}^*(it) + \delta_{\rho, i\infty} - \frac{3}{[SL_2(\Z):\Gamma_0(N)] \pi t} \right) dt.
		\end{equation}
		By the Weil bound, \[|j_{N,m}(\rho)| \le m\sum_{N|c\ge 1} \frac{\tau(c) \sqrt{c} \sqrt{\gcd(m,c)}}{c^2} < m^{3/2}\sum_{c\ge 1}\frac{\tau(c)}{c^{3/2}} = m^{3/2}\zeta(3/2)^2 \ll m^{3/2}.\]
		Thus \[\sum_{m\ge 1} |j_{N,m}(\rho) e^{-2\pi m t}| \ll \sum_{m\ge 1} m^{3/2} e^{-2\pi m t} =  \mr{Li}_{-3/2}(e^{-2\pi t}),\]
		which implies that convergence is absolute and locally uniform in $t$. Also, since $\mr{Li}_{-3/2}(e^{-2\pi t})$ is integrable over $[t_0, \infty)$ for $t_0>0$, it follows from the dominated convergence theorem that \[I^{(1)}_{\rho}(s,z;t_0) = \ell_\rho^{s-1} \sum_{m\ge 1} j_{N,m}(\rho) \int_{t_0/\ell_\rho}^\infty t^{s-1} e^{-2\pi mt}dt = \ell_\rho^{s-1} \sum_{m\ge 1} j_{N,m}(\rho)\frac{\Gamma(s, 2\pi m t_0 / \ell_\rho)} {(2\pi m)^s}.\]
		For fixed $t_0$, using $\Gamma(s, x) \sim x^{s-1} e^{-x}$ as $x\to \infty$ \cite[c.f.][{}6.5.32]{handbook_fns} gives
		\[\sum_{m\ge 1} \left|\ell_\rho^{s-1} j_{N,m}(\rho) \frac{\Gamma(s, 2\pi m t_0/\ell_\rho)} {(2\pi m)^s} \right| \ll \sum_{m\ge 1} \sqrt m t_0^{\sigma-1} e^{-2\pi m t_0/\ell_\rho} = t_0^{\sigma-1} \mr{Li}_{-1/2}(e^{-2\pi t_0/\ell_\rho}); \]
		hence convergence is uniform for $s$ in right-half planes. In particular $I^{(1)}_\rho(s,z;t_0)$ is entire in $s$ for fixed $t_0$. Incidentally, this also establishes the limit as $t_0\to \infty$ since $\mr{Li}_p(e^{-w})$ decays exponentially as $w\to \infty$. 
		
		Finally, using $|\Gamma(s, 2\pi m t_0)| \le \Gamma(\sigma, 2\pi mt_0) \le \Gamma(\sigma)$, we have 
		\[\sum_{m\ge 1} \left|j_{N,m}(\rho) \frac{\Gamma(s, 2\pi m t_0)} {(2\pi m)^s} \right| \ll \frac{\Gamma(\sigma)}{(2\pi)^\sigma} \sum_{m\ge 1} \frac{1}{m^{\sigma-3/2}}\]
		which converges for $\sigma>5/2$. Thus we have uniform convergence in $t_0$, which implies the limit as $t_0\to 0^+$. 
	\end{proof}
	
	Since we allow $t_0$ to depend on $y$, we need to add a minor restriction on $t_0$ to ensure convergence of $I^{(2)}_\rho$ and $I^{(3)}_\rho$. We give the following definition.
	\begin{definition}
		Suppose $f:\R^+\to \C$ is some function. We say that $f(t)$ depends \emph{rationally} on $t$ as $t\to a$ if there exists an integer $k\in \Z$ such that $f(t) = \Theta(t^k)$ as $t\to a$, i.e. there exist constants $B>A>0$ such that $At^k < |f(t)| < Bt^k$ for $t$ sufficiently close to $a$.
	\end{definition} 
	Essentially, $f(t)$ depends rationally on $t$ if it grows/decays as a rational function of $t$. Note that according to this definition, constant functions also depends rationally on $t$. 
	
	For $I^{(2)}_\rho(s,z;t_0)$, we have the following result.
	\begin{lemma}\label{lem::I2_convergence}
		Suppose $y>2$, and that $t_0$ depends rationally on $y$; moreover suppose there exists $k\in \N$ and $C\ge1$ such that $t_0\gg y^{-k}$ for all $y>C$. Then, the integral $I^{(2)}_\rho(s,z;t_0)$ converges for all $s\in \C$ and $y>C$, and defines an entire function of $s$. Moreover \[\lim_{y\to \infty} I^{(2)}_\rho(s,z;t_0) = 0\]
		even if $t_0$ depends rationally on $y$. 
	\end{lemma}
	\begin{remark}
		Notice that if $t_0$ is independent of $y$, then we can find $k\in \N$ such that $t_0 \gg 2^{-k}$, and so the proposition holds for all $y>2$. Hence this result shows that if $t_0$ is independent of $y$, then $I^{(2)}_\rho(s,z;t_0)$ is entire and has the required convergence properties uniformly in $y>2$. This result also holds for $t_0 = 1/y$.
	\end{remark}
	\begin{proof}
		The proof requires four lemmas. To state the lemmas, first notice that $|I_1(x)| = O(x)$ as $x\to 0$ and $|I_1(x)| = O(x^{-1/2} e^x)$ as $x\to \infty$ \cite[][{}(9.6.7) and (9.7.1)]{handbook_fns}. Fix $C>0$ such that $|I_1(x)| \ll x$ for $0<x<C$ and $|I_1(x)| \ll x^{-1/2} e^x$ for $x>C$, for some uniform bounds depending only on $C$. A numerical calculation shows that $C=2$ suffices, for instance. For $\rho \in \{i\infty, 0\}$, the cusp widths are $\ell_\rho = 1$ if $\rho = i\infty$ and $\ell_\rho = N$ if $\rho = 0$. The lemmas can now be stated as follows.
		
		\begin{lemma}\label{lem::I-bessel_piece_X1}
			For any bound $a = a(y)>0$ depending rationally on $y$, the following integral \[I_1(a) := \int_{a}^\infty t^{s-1}\sum_{m\ge 1} \sum_{\begin{smallmatrix} n,c\ge 1 \\ \sqrt{mn} < C \sqrt{\ell_\rho} c/4\pi\end{smallmatrix}} \sqrt{\frac{m}{n}} \frac{K_{i\infty, \rho}(m,-n;c)}{c} I_1\left(\frac{4\pi \sqrt{mn}}{\sqrt{\ell_\rho} c} \right) e^{2\pi i n z} e^{-2\pi m t} dt\]
			exists, is an entire function of $s$ for fixed $z$ and $t_0$, and for fixed $s$ converges to 0 as $y\to \infty$. 
		\end{lemma}
		
		\begin{lemma}\label{lem::I-bessel_piece_X2_n_ge_m}
			Suppose $y>2$. For any bound $a=a(y)>0$ depending rationally on $y$, the following integral \[I_2(a) := \int_{a}^\infty t^{s-1}\sum_{\begin{smallmatrix} m,c\ge 1, \ell_\rho n \ge m \\ \sqrt{mn} > C \sqrt{\ell_\rho} c/4\pi\end{smallmatrix}} \sqrt{\frac{m}{n}} \frac{K_{i\infty, \rho}(m,-n;c)}{c} I_1\left(\frac{4\pi \sqrt{mn}}{\sqrt{\ell_\rho} c} \right) e^{2\pi i n z} e^{-2\pi m t} dt\]
			exists, is an entire function of $s$ for fixed $z$ and $t_0$, and for fixed $s$ converges to 0 as $y\to \infty$. 
		\end{lemma}
		
		\begin{lemma}\label{lem::I-bessel_piece_X2_n_le_m_int_from_2<a_to_y}
			Suppose $y>0$. Let $a = a(y)$ be any bound that depends rationally on $y$ such that \[\inf_{y>0} a(y) > \frac2{\ell_\rho}.\] 
			Then \[I_3(a) := \int_{a}^\infty t^{s-1} \sum_{\begin{smallmatrix} n,c\ge 1, m>\ell_\rho n \\ \sqrt{mn} > C \sqrt{\ell_\rho} c/4\pi\end{smallmatrix}} \sqrt{\frac{m}{n}} \frac{K_{i\infty, \rho}(m,-n;c)}{c} I_1\left(\frac{4\pi \sqrt{mn}}{\sqrt{\ell_\rho} c} \right) e^{2\pi i n z} e^{-2\pi m t} dt\]
			exists, is an entire function of $s$ for fixed $z$ and $t_0$, and for fixed $s$ converges to 0 as $y\to \infty$. 
		\end{lemma}
		
		
		\begin{lemma}\label{lem::I-bessel_piece_X2_n_le_m_int_from_1/y to a}
			Suppose $y>2$. Consider bounds $a = a(y), b=b(y)$ both depending rationally on $y$ such that $0<a<b$, and moreover suppose that there exist $A>0$ and $k\in \N$ such that $a\ge2Ay^{-k}$ for all $y>1$. Then the following integral \[I_4(a,b) = \int_a^b t^{s-1} \sum_{\begin{smallmatrix} n,c\ge 1, m>\ell_\rho n \\ \sqrt{mn} > C\sqrt{\ell_\rho} c/4\pi\end{smallmatrix}} \sqrt{\frac{m}{n}} \frac{K_{i\infty, \rho}(m,-n;c)}{c} I_1\left(\frac{4\pi \sqrt{mn} }{\sqrt{\ell_\rho} c} \right) e^{2\pi i n z} e^{-2\pi m t/\ell_\rho} dt\]
			exists, is an entire function of $s$ for fixed $z$ and $t_0$, and for fixed $s$ converges to 0 as $y\to \infty$. 
		\end{lemma}
		
		To establish the proposition, it suffices to notice that if $t_0$ is bounded away from $2/\ell_\rho$, then $$I^{(2)}_\rho(s,z;t_0) = I_1(t_0) + I_2(t_0) + I_3(t_0),$$
		and otherwise 
		$$I^{(2)}_\rho(s,z;t_0) = I_1(t_0) + I_2(t_0) + I_3(c) + I_4(t_0, c)$$
		for some fixed $c>2/\ell_\rho$, where the desired convergence properties follow from the lemmas.
	\end{proof}
	We postpone the proof of the four lemmas to the end of the section due to its lengthy and tedious nature.
	
	Next, let us establish the convergence properties of $I^{(3)}_\rho(s,z;t_0)$.
	
	\begin{lemma}\label{lem::J-Bessel_converge_0}
		Suppose $y>1$, and suppose $t_0$ depends rationally on $y$. Then $I^{(3)}_\rho(s,z;t_0)$ exists and is an entire function of $s$, and converges to zero as $y\to \infty$. 
	\end{lemma}
	\begin{proof}
		Using $|J_1(x)| = O(x)$ for small $x$, $|J_1(x)| = O(1)$ for large $x$, and the Weil bound, we have
		\begin{align*}
			\left|\sum_{c\ge 1} \frac{K_{i\infty, \rho}(m,n;c)}{c} J_1 \left(\frac{4\pi \sqrt{mn}}{c} \right)\right| & \ll_{\ell_\rho} \sum_{c\ge1} \frac{\tau(c) \sqrt{\gcd(m,n,c)} \sqrt{c}}{c}\cdot \frac{\sqrt{mn}}{c} \ll n\sqrt m.
		\end{align*} 
		Thus the sum over $n$ and $c$ is bounded above by\[\ll_{\ell_\rho} \sum_{n\ge 1} \sqrt{\frac mn}\cdot  n\sqrt{m} e^{-2\pi n y}e^{-2\pi m t} = me^{-2\pi m t}\mr{Li}_{-1/2}(e^{-2\pi y}),\]
		and so the entire sum is bounded above by
		\begin{align*}
			&\ll_{\ell_\rho} \mr{Li}_{-1/2}(e^{-2\pi y})\sum_{m\ge 1} m e^{-2\pi m t} = \frac{\mr{Li}_{-1/2}(e^{-2\pi y}) e^{-2\pi t}}{(1 - e^{-2\pi t})^2}.
		\end{align*}
		Thus the required integral is bounded above by
		\begin{align*}
			&\ll_{\ell_\rho} \mr{Li}_{-1/2}(e^{-2\pi y}) \int_{t_0}^\infty \frac{t^{\sigma - 1}e^{-2\pi t}}{(1 - e^{-2\pi t})^2} dt \le \frac{\mr{Li}_{-1/2}(e^{-2\pi y}) }{(1 - e^{-2\pi t_0})^2} \frac{\Gamma(\sigma, 2\pi t_0)}{(2\pi)^\sigma}
		\end{align*}
		Thus the integral exists and we have local uniform convergence for $s\in \C$. Moreover, as $(1 - e^{-2\pi t_0})^{-2} = O(t_0^2)$ where $t_0$ depends rationally on $y$, it follows that $(1 - e^{-2\pi t_0})^{-2}$ depends rationally on $y$. The exponential decay of the polylogarithm term implies that the above bound goes to zero as $y\to \infty$. 
	\end{proof}
	
	Finally we consider $I^{(4)}$. This consists of two pieces, one of which is easily controlled.
	\begin{lemma}
		For any $t_0>0$ depending rationally on $y$,\[\int_{t_0}^\infty t^{s-1} \frac{1}{e^{2\pi (t + i\bar z)}-1} dt\] is entire in $s$, and converges to 0 as $y\to \infty$. 
	\end{lemma}
	\begin{proof}
		First note that \[\frac{1}{e^{2\pi (t + i\bar z)}-1} = \sum_{n\ge 1}e^{-2\pi i n \bar z} e^{-2\pi n t} = \sum_{n\ge 1}e(-nx) e^{-2\pi n (t+y)}.\]
		By dominated convergence in $t$, we have
		\[\int_{t_0}^\infty t^{s-1} \sum_{n\ge 1} e^{-2\pi i n \bar z} e^{-2\pi n t} dt = \sum_{n\ge 1} e^{-2\pi i n \bar z} \int_{t_0}^\infty t^{s-1} e^{-2\pi n t} dt = \sum_{n\ge 1} e^{-2\pi i n \bar z} \frac{\Gamma(s, 2\pi n t_0)}{(2\pi n)^s}\]
		For fixed $y$ and $t_0$, using $|\Gamma(s, 2\pi n t_0)| \le \Gamma(\sigma, 2\pi t_0)$ yields \[\left|\int_{t_0}^\infty t^{s-1} \frac{1}{e^{2\pi (t + i\bar z)}-1} dt \right| \le \frac{\Gamma(s, 2\pi t_0)}{(2\pi)^s} \mr{Li}_\sigma(e^{-2\pi y})\]
		and hence the series converges absolutely and locally uniformly in $s\in \C$ for fixed $y$ and $t_0$. Also, since $\Gamma(s, 2\pi t_0)$ decays exponentially as $t_0\to \infty$, is bounded for $t_0\to 0$ if $s\notin \Z_{\le 0}$, and grows polynomially for $t_0\to 0$ if $s\in \Z_{\le 0}$, the lemma follows.
	\end{proof}
	
	We also have the following lemma, which is easily seen due to the holomorphicity of $(e^{2\pi i(z-\tau)} - 1)^{-1}$ for $z\ne \tau$.
	\begin{lemma}
		For fixed $y$ and $t_0$, and assuming $x\notin \Z$, \[\int_{t_0}^\infty  \frac{t^{s-1}}{e^{2\pi (t+iz)} - 1} dt \]
		is entire in $s$.
	\end{lemma}
	In particular, these two lemmas show that $I^{(4)}(s,z;t_0)$ is entire in $s$ for fixed $z\notin \mc S_N$ and $t_0>0$. 
	
	We can now prove Propositions \ref{prop::mellin_transform_z_converges_infty_only} and \ref{prop::mellin_transform_z_converges}.
	\begin{proof}[Proof of Propositions \ref{prop::mellin_transform_z_converges_infty_only} and \ref{prop::mellin_transform_z_converges}]
		Notice that the lemmas above show that $I^{(1)}_\rho(s;t_0)$, $I^{(2)}_\rho(s,z;t_0)$, $I^{(3)}_\rho(s,z;t_0)$, and $I^{(4)}(s,z;t_0)$ are entire functions of $s$ assuming $y$ and $t_0$ are fixed. Using the decompositions given in (\ref{eqn::decompose_L_N,z}) and (\ref{eqn::decompose_L_N,-1/Nz}), it follows that the integrals are well-defined meromorphic functions of $s$, whose only pole is the obvious one at $s=1$. The independence from $t_0$ can then be checked by differentiating the integrals defining $L_{N,z}(s)$ and $L_{N,-1/Nz}(s)$. Finally, the functional equation can be established by comparing (\ref{eqn::decompose_L_N,z}) and (\ref{eqn::decompose_L_N,-1/Nz}) with $t_0 = N^{-1/2}$.
	\end{proof}
	
	\subsection{Proof of Lemmas \ref{lem::I-bessel_piece_X1}-\ref{lem::I-bessel_piece_X2_n_le_m_int_from_1/y to a}}\label{sctn::I2_s,z,t0}
	We now prove the four lemmas. For all four lemmas, we use the Weil bound to get \[\left| \sqrt{\frac{m}{n}} \frac{K(m,-n;c)}{c} I_1\left(\frac{4\pi \sqrt{mn}}{c} \right) e^{2\pi i n z}\right| \le \sqrt{\frac{m}{n}} \frac{\tau(c) \sqrt{\gcd(m,n,c)}}{\sqrt c} \left|I_1\left(\frac{4\pi \sqrt{mn}}{c} \right)\right| e^{-2\pi n y} .\]
	
	Each of the four lemmas will now be proved by bounding the sum over $n$ and $c$ of the right hand side, followed by bounding the sum over $m$.
	
	\begin{proof}[Proof for Lemma \ref{lem::I-bessel_piece_X1}]
		Using the fact that $|I_1(x)|\ll x$ for $x<C$, the inner sum over $n$ and $c$ given in the statement of Lemma \ref{lem::I-bessel_piece_X1} can be bounded against
		\begin{align*}
			&\ll \sum_{\sqrt{mn}<C\sqrt{\ell_\rho}c/4\pi}\sqrt{\frac{m}{n}} \frac{\tau(c) \sqrt{\gcd(m,n,c)}}{\sqrt c}\left(\frac{4\pi \sqrt{mn}}{c \sqrt{\ell_\rho}} \right) e^{-2\pi n y} \\
			&\ll_{\ell_\rho} \sum_{\sqrt{mn}<C \sqrt{\ell_\rho} c/4\pi} m\sqrt{n} \frac{\tau(c)}{c^{3/2}} e^{-2\pi n y} \ll m \mr{Li}_{-1/2}(e^{-2\pi y}).
		\end{align*}
		
		Here, we used the fact that \[\sum_{c\ge 1} \frac{\tau(c)}{c^2} = \zeta\left(\frac32\right)^2\]
		is a constant.
		Thus the required integral is bounded against
		
		\begin{align*}
			&\ll_{\ell_\rho} \mr{Li}_{-1/2}(e^{-2\pi y}) \int_{a}^\infty t^{\sigma - 1} \sum_{m\ge 1} me^{-2\pi m t} dt = \mr{Li}_{-1/2}(e^{-2\pi y}) \int_{a}^\infty t^{\sigma - 1} \frac{e^{2\pi t}}{(e^{2\pi t} - 1)^2} dt \\
			&\ll \frac{\mr{Li}_{-1/2}(e^{-2\pi y})}{(e^{2\pi a} - 1)^2} \int_{a}^\infty t^{\sigma - 1} e^{-2\pi t} dt = \frac{\mr{Li}_{-1/2}(e^{-2\pi y})}{(e^{2\pi a} - 1)^2} \frac{\Gamma(\sigma, 2\pi a)}{(2\pi)^\sigma} 
		\end{align*}
		for $a\ge 1/y$. However, we have $\mr{Li}_{-1/2}(e^{-2\pi y}) = O(e^{-2\pi y})$ as $y\to \infty$, whereas $(e^{2\pi a} - 1)^{-2} = O(a^{-2})$ as $a \to 0$. We thus have locally uniform convergence in $s$. Also, since we have assumed that $a$ decays at most like a rational function, the result follows.
	\end{proof}
	
	For the remaining three lemmas, note that the expression inside the $I$-Bessel function is greater than $C$, and so the terms inside the sum may be bounded above by 
	\begin{align*}
		&\ll \sqrt{\frac{m}{n}} \frac{\tau(c) \sqrt{\gcd(m,n,c)}}{\sqrt c} \frac{\ell_\rho^{1/4}\sqrt c}{(mn)^{1/4}} \exp\left(\frac{4\pi \sqrt{mn}}{\sqrt{\ell_\rho} c} \right) e^{-2\pi n y} e^{-2\pi m t} \\
		&\ll_{\ell_\rho} \frac{m^{1/4}}{n^{3/4}} \min\{\sqrt m, \sqrt n\} \tau(c)\exp\left(\frac{4\pi \sqrt{mn}}{\sqrt{\ell_\rho}} \right)e^{-2\pi n y} e^{-2\pi m t}. 
	\end{align*}
	Thus, the sum over $c$ is bounded above by 
	\[\ll_{\ell_\rho} \frac{m^{1/4}}{n^{3/4}} \min\{\sqrt m, \sqrt n\}\exp\left(\frac{4\pi \sqrt{mn}}{\sqrt{\ell_\rho} c} \right)e^{-2\pi n y} e^{-2\pi m t} \sum_{c<4\pi \sqrt{mn}/C \sqrt{\ell_\rho}}\tau(c),\]
	but since $\sum_{c\le x} \tau(x) = O(x\log x)$, this is bounded above by
	\begin{align}
		&\ll_{\ell_\rho} \frac{m^{1/4}}{n^{3/4}} \min\{\sqrt m, \sqrt n\}\sqrt{mn} \log(mn) \exp \left(\frac{4\pi \sqrt{mn}}{\sqrt{\ell_\rho}}\right) e^{-2\pi n y} e^{-2\pi m t} \nonumber \\
		&\ll_{\ell_\rho} \frac{m^{3/4}}{n^{1/4}} \min\{\sqrt m, \sqrt n\}\log(mn) \exp \left(\frac{4\pi \sqrt{mn}}{\sqrt{\ell_\rho}}\right) e^{-2\pi n y} e^{-2\pi m t} \label{eqn::bound_terms_when_I1bessel_exp}
	\end{align}
	
	We now use this bound to prove the remaining three lemmas. 
	
	\begin{proof}[Proof for Lemma \ref{lem::I-bessel_piece_X2_n_ge_m}]
		Using the bound given in (\ref{eqn::bound_terms_when_I1bessel_exp}), the sum over $m,n,c$ is bounded against 
		\begin{align*}
			&\ll_{\ell_\rho} \sum_{m\le \ell_\rho n} \frac{m^{3/4}}{n^{1/4}} \sqrt{m} \log(mn) e^{4\pi \sqrt{mn}/\sqrt{\ell_\rho}}e^{-2\pi n y} e^{-2\pi m t} \\
			&\ll_{\ell_\rho, \delta} \sum_{n\ge 1} n^{1+\delta} e^{2\pi n (2- y)} \sum_{1\le m\le n} e^{-2\pi m t} \\
			&\ll_\delta \frac{1}{e^{2\pi t} - 1} \sum_{n\ge 1} n^{1+\delta} e^{-2\pi n(y-2)} = \frac{1}{e^{2\pi t} - 1} \mr{Li}_{-1-\delta}(e^{-2\pi (y-2)})
		\end{align*}
		for any $\delta>0$, which implies that the required integral is bounded against \[\ll_{\ell_\rho, \delta} \mr{Li}_{-1-\delta} (e^{-2\pi(y- 2)}) \int_{a}^\infty \frac{t^{\sigma - 1}}{e^{2\pi t} - 1} dt.\]
		We now need to bound this integral over $t$. If $a\ge 1$, then we have \[\int_{a}^\infty \frac{t^{\sigma - 1}}{e^{2\pi t} - 1} dt \le \frac{1}{1 - e^{-2\pi \cdot 1}} \int_a^\infty t^{\sigma-1} e^{-2\pi t} dt \le \frac{1}{1 - e^{-2\pi}} \frac{\Gamma(\sigma, 2\pi)}{(2\pi)^\sigma}. \]
		If $a<1$, then 
		\begin{align*}
			\int_{a}^\infty \frac{t^{\sigma - 1}}{e^{2\pi t} - 1} dt &\le \frac{1}{e^{2\pi a} - 1} \int_{a}^1t^{\sigma - 1}dt + \frac{1}{1 - e^{-2\pi}} \int_1^\infty t^{\sigma-1} e^{-2\pi t} dt \\
			&\le \frac{1}{1 - e^{-2\pi}}\frac{\Gamma(\sigma, 2\pi)}{(2\pi)^\sigma} + \frac{1}{e^{2\pi a} - 1} \cdot \begin{cases}
				\sigma^{-1} (1 - a^\sigma) & \text{if } \sigma \ne 0,\\
				-\log a & \text{if } \sigma = 0.
			\end{cases} 
		\end{align*}
		In any case, we have bounded the required integral by $\mr{Li}_{-1-\delta}(\exp(2\pi(2-y)))$ times something that grows at most polynomially in $y$ as $y\to \infty$, and moreover this bound is locally uniform in $s$. The result follows.
	\end{proof}
	
	\begin{proof}[Proof for Lemma \ref{lem::I-bessel_piece_X2_n_le_m_int_from_2<a_to_y}]
		The sum over $m,n,c$ is bounded against 
		\begin{align*}
			&\ll_{\ell_\rho} \sum_{m>\ell_\rho n} \frac{m^{3/4}}{n^{1/4}} \sqrt n e^{4\pi m/\ell_\rho} e^{-2\pi n y}e^{-2\pi m t} \log{mn} \\
			&\ll_{\ell_\rho, \delta} \sum_{m\ge 1} m^{1+\delta} e^{2\pi m(2/\ell_\rho-t)} \sum_{1\le n\le m/\ell_\rho} e^{-2\pi n y} 
			\ll_\delta \frac1{e^{2\pi y} - 1} \sum_{m\ge 1} m^{1+\delta} e^{2\pi m(2/\ell_\rho-t)}
		\end{align*}
		for any $\delta>0$. Since the above series is a polylogarithm, which is integrable for $t\in [a, \infty]$, dominated convergence implies that the required integral is bounded above by 
		\begin{align*}
			&\ll_{\ell_\rho, \delta} \frac1{e^{2\pi y} - 1} \sum_{m\ge 1} m^{1+\delta} \int_a^\infty t^{\sigma-1} e^{-2\pi m(t- 2/\ell_\rho)} dt 
			= \frac1{e^{2\pi y} - 1} \sum_{m\ge 1} m^{1+\delta} e^{4\pi m/\ell_\rho} \frac{\Gamma(\sigma, 2\pi m a)}{(2\pi m)^\sigma} \\
			& \ll \frac1{e^{2\pi y} - 1} \sum_{m\ge 1}  m^\delta a^{\sigma - 1} e^{-2\pi m (a-2/\ell_\rho)} 
			= \frac1{e^{2\pi y} - 1} a^{\sigma - 1}\mr{Li}_{-\delta}(e^{-2\pi (a-2/\ell_\rho)}).
		\end{align*}
		Here, we used the fact that $\Gamma(\sigma, x) \sim x^{\sigma - 1} e^{-x}$ as $x\to \infty$. The lemma follows as before. 
	\end{proof}
	
	\begin{proof}[Proof for Lemma \ref{lem::I-bessel_piece_X2_n_le_m_int_from_1/y to a}]
		We fix $k\in \N$ and $A>0$ such that $a \ge 2Ay^{-k}$. Since $y>2$, we may fix $k$ large enough so that $y^{k+1}>2A$ for all $y>2$. Then, by (\ref{eqn::bound_terms_when_I1bessel_exp}), and the fact that $n<m$, the required integral is bounded above by 
		\begin{equation}\label{eqn::I1bessel_integral_1/y-a_splitting}
			\begin{split}
				\ll_{\ell_\rho} \int_{a}^b t^{\sigma-1} \sum_{n\ge 1, m> (y^{k+2}/ A)^2\ell_\rho n}& m\log m e^{4\pi \sqrt{mn/\ell_\rho}} e^{-2\pi n y}e^{-2\pi m t/\ell_\rho} dt \\
				&+ \int_a^b t^{\sigma-1} \sum_{n\ge 1, (y^{k+2}/ A)^2 \ell_\rho n > m> \ell_\rho n} m \log m e^{4\pi \sqrt{mn/\ell_\rho}} e^{-2\pi n y}e^{-2\pi m t/\ell_\rho} dt.
			\end{split}
		\end{equation}
		The integrand for the first integral is
		\begin{align*}
			&= t^{\sigma-1} \sum_{n\ge 1} e^{-2\pi n y} \sum_{m> (y^{k+2}/ A)^2 \ell_\rho n} m \log m e^{-2\pi m (t - 2\sqrt{\ell_\rho n/m})/\ell_\rho} \\
			&\ll_\delta t^{\sigma-1} \sum_{n\ge 1} e^{-2\pi n y} \sum_{m> (y^{k+2}/ A)^2\ell_\rho n} m^{1+\delta} e^{-2\pi m (t - 2A/y^{k+2})/\ell_\rho} \\
			&\ll_\delta \frac{t^{\sigma - 1}}{e^{2\pi y} - 1} \mr{Li}_{-1-\delta}(e^{-2\pi (t - 2A/y^{k+2})/\ell_\rho})
		\end{align*}
		for any $\delta>0$. Thus, the first integral in (\ref{eqn::I1bessel_integral_1/y-a_splitting}) is bounded above by 
		\[\ll_\delta \int_a^b \frac{t^{\sigma - 1}}{e^{2\pi y} - 1} \mr{Li}_{-1-\delta}(e^{-2\pi (t - 2A/y^{k+2})/\ell_\rho}) dt
		\ll \frac{(b-a) \max\{a^{\sigma - 1}, b^{\sigma-1}\}}{e^{2\pi y} - 1}  \mr{Li}_{-1-\delta}\left( \exp \left( -\frac{2\pi (y^{k+1} - 2A)}{\ell_\rho y^{k+2}}\right) \right),\]
		where we use the fact that $\mr{Li}_p(e^{-w})$ is a decreasing function of $w$ (for $w>0$). Now, from (9.3) of \cite{wood_polylogs}, we have \[\mr{Li}_{-1-\delta}(e^{-w}) = O_\delta(w^{-2-\delta})\]
		as $w\to 0$, and so the above bound is essentially a rational function in $y$ times $(e^{2\pi y} - 1)^{-1}$, and so goes to zero as $y\to \infty$.
		
		Finally, for the remaining piece, note that for any $\delta>0$ we have
		\begin{align*}
			Y(t) &:= \sum_{n\ge 1, (y^{k+2}/A)^2 \ell_\rho n > m>\ell_\rho n} m \log m e^{4\pi \sqrt{mn/\ell_\rho}} e^{-2\pi n y}e^{-2\pi m t/\ell_\rho} \\
			&\ll_{\delta, \ell_\rho} \left( \frac{y^{2k+4}}{A^2}\right)^{1 + \delta} \sum_{n\ge 1} n^{1+\delta} \sum_{(y^{k+2}/ A)^2 \ell_\rho n > m>\ell_\rho n} e^{4\pi \sqrt{mn/\ell_\rho}} e^{-2\pi n y} e^{-2\pi m t/\ell_\rho} \\
			&= \left( \frac{y^{k+2}}{A}\right)^{2 + 2\delta} \sum_{n\ge 1} n^{1+\delta} \sum_{(y^{k+2}/ A)^2\ell_\rho n > m>\ell_\rho n} e^{-2\pi (\sqrt{m/\ell_\rho} - \sqrt n)^2} e^{-2\pi n (y-1)} e^{2\pi m (1 - t)/\ell_\rho} \\
			&\le \left( \frac{y^{k+2}}{\epsilon}\right)^{2 + 2\delta} \sum_{n\ge 1} n^{1+\delta} e^{-2\pi n (y-1)}  e^{-2\pi n ((y^{k+2}/A) - 1)^2}\sum_{(y^{k+2}/ A)^2\ell_\rho n > m>\ell_\rho n} e^{2\pi m (1 - t)/\ell_\rho} .
		\end{align*}
		To handle this last geometric sum, we need to split the integral into two parts; however, if $b\le1$ then the first part is not necessary, while if $a\ge1$ then the second part is not necessary
		\begin{itemize}
			\item For $t>1$, notice that we can bound by 
			\begin{align*}
				Y(t) &\ll_{\delta, \ell_\rho} \left( \frac{y^{k+2}}{A}\right)^{2+ 2\delta} \sum_{n\ge 1} n^{1+\delta} e^{-2\pi n (y-1)}  e^{-2\pi n ((y^{k+2}/A) - 1)^2} \cdot \frac{y^{2k+4}}{A^2}n e^{-2\pi (t-1)} \\
				&= \left( \frac{y^{k+2}}{A}\right)^{4 + 2\delta} e^{-2\pi (t-1)} \mr{Li}_{-2-\delta} \left(\exp\left(-2\pi (y - 1) - 2\pi \left(\frac {y^{k+2}}A - 1 \right)^2 \right) \right)
			\end{align*}
			and so
			\begin{align*}
				\int_{[a,b]\cap[1,\infty)} t^{\sigma - 1} Y(t) dt &\ll_{\delta, \ell_\rho} \left( \frac{y^{k+2}}{A}\right)^{4 + 2\delta}  \mr{Li}_{-2-\delta} \left(\exp\left(-2\pi (y - 1) - 2\pi \left(\frac {y^{k+2}}A - 1 \right)^2 \right) \right) \int_1^b t^{\sigma - 1} e^{-2\pi (t-1)} dt \\
				&\ll \left( \frac{y^{k+2}}{A}\right)^{4 + 2\delta} (b-1)\max\{1, b^{\sigma - 1}\} \cdot \mr{Li}_{-2-\delta} \left(\exp\left(-2\pi (y - 1) - 2\pi \left(\frac {y^{k+2}}A - 1 \right)^2 \right) \right)
			\end{align*}
			which goes to 0 as $y\to \infty$.
			
			\item For $t<1$, we sum the geometric series directly to get
			\begin{align*}
				Y(t) &\ll_{\delta,\ell_\rho } \left( \frac{y^{k+2}}{A} \right)^{2 + 2\delta} \sum_{n\ge 1} n^{1+\delta} e^{-2\pi n (y-1)}  e^{-2\pi n ((y^{k+2}/A) - 1)^2} \cdot \frac{e^{2\pi(1-t)n y^{2k+4}/A^2 } - e^{2\pi n(1-t)}}{e^{2\pi(1-t)} - 1} \\
				\begin{split}
					&\ll \frac{1}{e^{2\pi(1-t)} - 1} \left( \frac{y^{k+2}}{A}\right)^{2 + 2\delta} \Bigg(\sum_{n\ge 1} n^{1+\delta} \exp\left(-2\pi n (y-1) -2\pi n \left( \frac {y^{k+2}} A - 1\right)^2 + 2\pi(1-t)n \frac{y^{2k+4}}{A^2}\right) \\
					&\qquad \qquad \qquad \qquad \qquad \quad - \sum_{n\ge 1} n^{1+\delta} \exp\left(-2\pi n (y-1) -2\pi n \left( \frac {y^{k+2}} A - 1\right)^2 + 2\pi n(1-t)\right)\Bigg) 
				\end{split}\\
				&= \frac{\left(y^{k+2}/A\right)^{2 + 2\delta}}{e^{2\pi(1-t)} - 1} \left(\mr{Li}_{-1 - \delta}\left(e^{-2\pi(y - \frac{2}{A} y^{k+2} + \frac{1}{A^2} ty^{2k+4})} \right) - \mr{Li}_{-1-\delta}\left(e^{-2\pi(y - \frac{2}{A} y^{k+2} + \frac{1}{A^2} y^{2k+4} - (1-t))} \right) \right).
			\end{align*}
			Here, notice that our assumptions on $y$, $a$, $k$, and $A$ implies that for $t\ge a\ge 2Ay^{-k}$, \[y - \frac{2}{A} y^{k+2} + \frac{1}{A^2} ty^{2k+4} \ge y - \frac{2}{A} y^{k+2} + \frac{2}{A} y^{k+4} > 0\]
			and \[y - \frac{2}{A} y^{k+2} + \frac{1}{A^2} y^{2k+4} - (1-t) > (y-1) \frac{(y^{k+2} - 2A)y^{k+2}}{A^2}>0,\]
			and so the above polylogarithms are indeed well-defined. Moreover, notice that the simple zero at $t=1$ of $e^{2\pi(1-t)}-1$ in the denominator cancels with the zero in the numerator (the two polylogarithm terms are equal at $t=1$), and so the above expression is bounded as $t\to 1^-$. Thus the above expression is in fact integrable over $[a, 1]$, and so by dominated convergence 
			\begin{align*}
				\int_{a}^1 t^{\sigma - 1} &Y(t) dt \ll_{\delta, \ell_\rho} \left( \frac{y^{k+2}}{A}\right)^{2 + 2\delta} \sum_{n\ge 1} n^{1+\delta} e^{-2\pi n (y-1)}  e^{-2\pi n ((y^{k+2}/A) - 1)^2}  \\
				&\qquad \qquad \qquad \qquad \qquad \qquad \qquad \cdot\sum_{(y^{k+2}/ A)^2\ell_\rho n > m>\ell_\rho n} \int_a^1 t^{\sigma - 1} e^{2\pi m (1 - t)/\ell_\rho} dt \\
				&\ll \left( \frac{y^{k+2}}{A}\right)^{2 + 2\delta} \sum_{n\ge 1} n^{1+\delta} \exp\left(-2\pi n \left(y-\tfrac{2y^{k+2}}A + \tfrac{y^{2k+4}}{A^2} \right)\right)   \\
				&\qquad \qquad \qquad \qquad \qquad \qquad \qquad \cdot \sum_{(y^{k+2}/ A)^2 n \ell_\rho > m>\ell_\rho n} \left(1 - a\right) \max\{1, a^{\sigma-1}\} e^{2\pi m (1 - a)/\ell_\rho} \\
				&\ll_{\ell_\rho} \left( \frac{y^{k+2}}{A}\right)^{2 + 2\delta}(1-a)\max\{1, a^{\sigma-1}\}  \sum_{n\ge 1} n^{1+\delta} \exp\left(-2\pi n \left(y-\tfrac{2y^{k+2}}A + \tfrac{y^{2k+4}}{A^2} \right)\right)  \\
				&\qquad \qquad \qquad \qquad\qquad \qquad \qquad \qquad \qquad \qquad\cdot e^{2\pi (y^{k+2}/A)^2n (1 - a)} \frac{y^2}{\epsilon^2} n\\
				&= \left( \frac{y^{k+2}}{A}\right)^{4 + 2\delta} (1-a) \max\{1, a^{\sigma-1}\}  \sum_{n\ge 1} n^{2+\delta} \exp\left(-2\pi n \left(y-2\frac{y^{k+2}}A + \frac{y^{2k+4}}{A^2 }- \frac{y^{2k+4}}{A^2}(1 - a)\right) \right) \\
				&= \left( \frac{y}{\epsilon}\right)^{4 + 2\delta} \left(1 - \frac1y \right) \max\{1, y^{1-\sigma}\} \mr{Li}_{-2-\delta} \left(\exp\left(-2\pi \left(y-2\frac{y^{k+2}}A + a\frac{y^{2k+4}}{A^2 } \right) \right) \right)
			\end{align*}
			and this bound goes to 0 as $y \to \infty$. 
		\end{itemize} 
		This establishes the lemma.
	\end{proof}

	\section{Proof of Proposition \ref{prop::mellin_transform_at_cusps}}\label{sctn::continue_series_jNm_infty}
	\begin{proof}[Proof of Proposition \ref{prop::mellin_transform_at_cusps} (1)-(3)]
		If these integrals are well-defined, then the independence from $t_0$ follows at once by simply differentiating the two expressions. To prove convergence of the integrals, we use the substitution $t\mapsto \frac1{Nt}$ followed by the functional equation given in Corollary \ref{cor::eisenstein_func_eqn} and then (\ref{eqn::I1_as_regularized_mellin}) to get
		\begin{align*}
			\int_{0}^{t_0} t^{s-1} \left(H_{N,i\infty }^*(it) + \frac{3}{c_N\pi t} \right)dt &= N^{1-s} \int_{1/Nt_0}^\infty t^{1-s} \left(-Nt^2H_{N,0}^*(it) + \frac{3Nt}{c_N\pi } \right)\frac{dt}{Nt^2} = - I_0^{(1)} \left(2-s;\frac1{t_0} \right) \\
			\int_{0}^{t_0} t^{s-1} \left(H_{N,0 }^*(it) - \frac1{Nt^2} + \frac{3}{c_N\pi t} \right)dt &= N^{1-s} \int_{1/Nt_0}^\infty t^{1-s} \left(-Nt^2H_{N,i\infty}^*(it) - Nt^2 + \frac{3Nt}{c_N\pi } \right) \frac{dt}{Nt^2} \\
			&= -N^{1-s} I_{i\infty}^{(1)} \left(2-s;\frac1{Nt_0} \right),
		\end{align*}
		where $c_N = [SL_2(\Z):\Gamma_0(N)]$. Hence 
		\begin{align}
			L_{N, i\infty}(s) &= \frac{t_0^s}{s} - \frac{6}{c_N\pi} \frac{t_0^{s-1}}{s-1} -I_0^{(1)} \left(2-s;\frac1{t_0} \right) + I_{i\infty}^{(1)}(s,t_0) \label{eqn::H_infty_L_function}\\
			L_{N, 0}(s) &= \frac1N \frac{t_0^{s-2}}{s-2} - \frac{6}{c_N\pi} \frac{t_0^{s-1}}{s-1} - N^{1-s}I_{i \infty }^{(1)} \left(2-s;\frac1{Nt_0} \right) + N^{1-s} I_{0}^{(1)}(s,Nt_0) \label{eqn::H_zero_L_function}.
		\end{align}
		Lemma \ref{lem::I1_convergence} implies (1) immediately. Part (2) follows by taking the limit as $t_0\to 0^+$ in (\ref{eqn::H_infty_L_function}) and (\ref{eqn::H_zero_L_function}), and using Lemma \ref{lem::I1_convergence}. Part (3) follows by setting $t_0 = 1$ in (\ref{eqn::H_infty_L_function}) and $t_0 = \frac1N$ in (\ref{eqn::H_zero_L_function}) and comparing expressions. 
	\end{proof}
	
	Let us now find the local factors in the Euler product of $L_{N,i\infty}(s)$ and $L_{N,0}(s)$. In light of Proposition \ref{prop::mellin_transform_at_cusps}(3), the factorization into an Euler product of $L_{N,0}(s)$ can easily be calculated from that of $L_{N, i\infty}(s)$. To find the factorization into an Euler product for $L_{N, i\infty}(s)$, we first require the following lemma, which follows easily by looking at the Euler product of each side. Here, $\zeta(s)$ is the usual Riemann zeta function.
	
	\begin{lemma}\label{lem::dirichletseries_co-prime_integers}
		For any $N\in \N$, we have \[\sum_{a\ge 1, \gcd(a, N)=1} \frac{1}{a^s} = \prod_{p\nmid N} \left(1 - p^{-s} \right) = \zeta(s) \prod_{p|N} \left(1 - p^{-s} \right)\]
	\end{lemma}
%
	\begin{proof}[Proof of Proposition \ref{prop::mellin_transform_at_cusps}(4)]
		First note that 
		\begin{align*}
			j_{N,m}(i\infty) &= 4\pi^2 m \sum_{N|c\ge 1} \frac{K(m,0;c)}{c^2} = 24m \left(\sum_{k\ge 1} \frac{1}{k^2} \right) \left( \frac{1}{N^2} \sum_{k \ge 1}  \frac{K(m,0;Nk)}{k^2}\right) = \frac{24m}{N^2} \sum_{k\ge1} \frac{1}{k^2} \sum_{d|k} c_{Nd}(m)
		\end{align*}
		where $c_q(m) = K(m,0;q)$ is Ramanujan's sum. Thus, by absolute convergence of the Dirichlet series for $\Re{s}> \frac52$, and using the Dirichlet series for $c_q(m)$ (given in \cite{ramanujan_sums}), we have
		
		\begin{align*}
			\sum_{m\ge 1} \frac{j_{N,m}(i\infty)}{m^s} &= \frac{24}{N^2} \sum_{k\ge 1} \frac{1}{k^2} \sum_{d|k} \sum_{m\ge 1} \frac{c_{Nd}(m)}{m^{s-1}} = \frac{24}{N^2} \sum_{k\ge 1} \frac{1}{k^2} \sum_{d|k} \zeta(s-1) \sum_{g|Nd} \mu(g) \left( \frac{Nd}{g}\right)^{2-s} \\
			&= 24 \zeta(s-1) N^{-s} \sum_{k\ge 1} \sum_{d|k} \frac{d^{2-s}}{k^2} \sum_{g|Nd} \mu(g) g^{s-2}
			= 24 \zeta(s-1) N^{-s} \sum_{d\ge 1} \sum_{k\ge 1} \frac{1}{k^2d^s} \sum_{g|Nd} \mu(g) g^{s-2} \\
			&= 24 \zeta(s-1) N^{-s} \left( \prod_{p} \frac{1}{1 - p^{-2}} \right) \sum_{d\ge 1} \frac{1}{d^s} \sum_{g|Nd} \mu(g) g^{s-2}.
		\end{align*}
		To evaluate the sum over $d$, we use the multiplicativity of $\mu$ to get 
		\begin{align*}
			\sum_{d\ge 1} \frac{1}{d^s} \sum_{g|Nd} \mu(g) g^{s-2} &= \sum_{d\ge 1} \frac{1}{d^s} \prod_{p|Nd} (\mu(1)1^{s-2} + \mu(p) p^{s-2}) = \sum_{d\ge 1} \frac{1}{d^s} \prod_{p|Nd} (1 - p^{s-2}) \\
			&= \left( \prod_{p|N} (1 - p^{s-2})\right) \sum_{d\ge 1} \frac{1}{d^s} \prod_{p|d, p\nmid N} (1 - p^{s-2}) \\
			&= \left( \prod_{p|N} (1 - p^{s-2})\right) \prod_{p| N}\left(\sum_{\ell \ge 0} \frac{1}{(p^\ell)^s} \right)\sum_{d'\ge 1, (d', N) = 1} \frac{1}{d'^s} \prod_{p|d'} (1 - p^{s-2}) \\
			&= \left( \prod_{p|N} (1 - p^{s-2})\right) \prod_{p| N}\left(\frac{1}{1 - p^{-s}} \right) \sum_{d'\ge 1, (d', N) = 1} \frac{1}{d'^s} \sum_{\delta|d'} \mu( \delta) \delta^{s-2}
		\end{align*}
		where we rewrote $d$ as $d = d' \prod_{p|N} p^{\ell(p)}$ with $\gcd(N,d')=1$. Interchanging the sums over $d'$ and $\delta$ then gives 
		\begin{align*}
			\sum_{d\ge 1} \frac{1}{d^s} \sum_{g|Nd} \mu(g) g^{s-2} &= \prod_{p| N}\left(\frac{1 - p^{s-2}}{1 - p^{-s}} \right) \sum_{\delta\ge 1, \gcd(\delta, N) = 1} \mu( \delta) \delta^{s-2} \sum_{\delta|d'\ge 1, \gcd(d', N)=1} \frac{1}{d'^s} \\
			&= \prod_{p| N}\left(\frac{1 - p^{s-2}}{1 - p^{-s}} \right) \sum_{\delta\ge 1, \gcd(\delta, N) = 1}  \frac{\mu( \delta)}{\delta^2} \sum_{a\ge 1, \gcd(a, N)=1} \frac{1}{a^s} 
		\end{align*}
		By Lemma \ref{lem::dirichletseries_co-prime_integers}, we have \[\sum_{d\ge 1} \frac{1}{d^s} \sum_{g|Nd} \mu(g) g^{s-2} = \zeta(s)\prod_{p| N}\left(1 - p^{s-2}\right) \sum_{\delta\ge 1, \gcd(\delta, N) = 1}  \frac{\mu( \delta)}{\delta^2} = \zeta(s)\prod_{p| N}\left(1 - p^{s-2}\right) \prod_{p\nmid N} \left(1 - \frac1{p^2} \right).\]
		Therefore we have 
		\begin{align*}
			\sum_{m\ge 1} \frac{j_{N,m}(i\infty)}{m^s} &= 24 \zeta(s-1) N^{-s} \left( \prod_{p} \frac{1}{1 - p^{-2}} \right) \zeta(s)\prod_{p| N}\left(1 - p^{s-2}\right)  \prod_{p\nmid N} \left(1- \frac{1}{p^2} \right) \\
			&= 24 \zeta(s) \zeta(s-1) N^{-s} \prod_{p| N}\frac{1 - p^{s-2}}{1 - p^{-2}}
		\end{align*}
		which is the required continuation. 
	\end{proof}
	
	\begin{proof}[Proof of Corollary \ref{cor::explicit_formula_for_j_N,n_rho}]
		The proof for $j_{N,n}(0)$ follows directly from 
		\[\sum_{n\ge 1} \frac{j_{N,n}(0)}{n^s} = \frac{24}{N} \prod_{p|N} \frac{1}{1-p^{-2}} \zeta(s)\zeta(s-1)\prod_{p|N} (1 - p^{-s}) = \frac{24}{N} \prod_{p|N} \frac{1}{1-p^{-2}} \left(\prod_{p\nmid N} \sum_{j\ge 0} \frac{\sigma(p^j)}{p^{js}} \right) \left(\prod_{p| N} \sum_{j\ge 0} \frac{p^j}{p^{js}} \right).\]
		
		We now consider $j_{N,n}(i\infty)$. Expanding the analytic continuation given in Proposition \ref{prop::mellin_transform_at_cusps}(4) as an Euler product, it is easy to see that 
		\begin{align*}
			\sum_{m\ge 1} \frac{j_{N,m}(i\infty)}{m^s} &= 24N^{-s} \left( \prod_{p|N}\frac{1}{1 - p^{-2}} \right)  \left( \prod_{p\nmid N} \sum_{j\ge 0} \frac{\sigma(p^j)}{p^{js}}\right) \left(\prod_{p|N} \left(-p^{s-2} + \sum_{j\ge 0} \frac{\sigma(p^j) - p^{-2} \sigma(p^{j+1})}{p^{js}} \right) \right) \\
			&= \frac{24}{N^s} \left( \prod_{p|N}\frac{1}{1 - p^{-2}} \right)  \left( \sum_{\begin{smallmatrix}m\ge 1,\\ \gcd(m,N)=1 \end{smallmatrix}} \frac{\sigma(m)}{m^{s}}\right) \left(\sum_{d|N} \mu(d) d^{s-2} \prod_{p|N, p\nmid d} \sum_{j\ge 0} \frac{\sigma(p^j) - p^{-2} \sigma(p^{j+1})}{p^{js}}\right)
		\end{align*}
		The rightmost expression times $N^{-s}$ can be rewritten as
		\begin{align*}
			&N^{-2}\sum_{d|N} \mu\left(\frac Nd\right)d^{2-s} \prod_{p|d, p\nmid N/d} \sum_{j\ge 0} \frac{\sigma(p^j) - p^{-2} \sigma(p^{j+1})}{p^{js}} \\
			=& N^{-2} \sum_{d|N} \mu\left(\frac Nd\right) d^2 \sum_{\begin{smallmatrix}m\ge 1, \mbb P(m) \subseteq \mbb P(d) \\ \gcd(m, N/d)=1 \end{smallmatrix}} \frac{1}{(dm)^s} \sum_{g|m} \frac{\mu(g)}{g^2} \sigma(mg)\\
			=& N^{-2} \sum_{d|N} \mu\left(\frac Nd\right) d^2 \sum_{\begin{smallmatrix}m\ge 1, \mbb P(m) \subseteq \mbb P(N) \\ \gcd(m, N)=d \end{smallmatrix}} \frac{1}{m^s} \sum_{g|m/d} \frac{\mu(g)}{g^2} \sigma\left( \frac{mg}{d}\right) \\
			=& N^{-2} \sum_{m\ge 1, \mbb P(m) \subseteq \mbb P(N)} \frac{1}{m^s} \mu\left(\frac{N}{\gcd(m,N)}\right) \gcd(N,m)^2 \sum_{g|m/\gcd(m,N)} \frac{\mu(g)}{g^2} \sigma \left(\frac{mg}{\gcd(m,N)} \right),
		\end{align*} 
		where $\mbb P(n)$ is the set of prime factors of $n$. Comparing coefficients of the two Dirichlet series yields \[j_{N,n_1n_2}(i\infty) = \frac{24}{N^2} \left( \prod_{p|N}\frac{1}{1 - p^{-2}} \right) \sigma(n_1) \mu\left(\frac{N}{\gcd(n_2,N)}\right) \gcd(N,n_2)^2 \sum_{g|n_2/\gcd(n_2,N)} \frac{\mu(g)}{g^2} \sigma \left(\frac{n_2g}{\gcd(n_2,N)} \right).\]
		where $\gcd(n_1, n_2) = 1 = \gcd(n_1, N)$ and $\mbb P(n_2) \subseteq \mbb P(N)$. The multiplicativity of $\sigma$, and the fact that $\mu(g)$ is non-zero if and only if $g$ is square-free, then yields the corollary. 
	\end{proof}
		
	\section{Proof of Theorem \ref{thm::mellin_transform_converge_limit}}\label{sctn::proof_thm}

	\begin{proof}[Proof of Theorem \ref{thm::mellin_transform_converge_limit}]
		Using the decomposition of $L_{N,z}$ given in (\ref{eqn::decompose_L_N,z}) as well as the decomposition of $L_{N,i\infty}$ given in (\ref{eqn::H_infty_L_function}) with $t_0 = 1/y$, and then noticing that the $I^{(2)}$ and $I^{(3)}$ pieces are entire and go to zero as $y\to \infty$, we have \[L_{N,z}(s) = L_{N,i\infty}(s) - \frac{y^{-s}}{s} + I^{(4)}\left(s,z;\frac1y \right) + o_s(1).\]
		Thus it only remains to study the behaviour of $I^{(4)}(s, x+iy; 1/y)$ ($x\notin \Z$) as $y\to \infty$. We show that, for all $s\in \C$,
		\begin{equation}\label{eqn::exponential_int_1/y_to_infty}
			I^{(4)}\left(s,z, \frac1y \right) = \frac{y^{-s}-y^{s}}{s} + \sum_{j=1}^{\lfloor \sigma \rfloor} \frac{(s-1)_{j-1} y^{s-j}}{(2\pi)^j} \Big(\mr{Li}_j\big(e^{-2\pi i x}\big) + (-1)^j\mr{Li}_j\big( e^{2\pi i x}\big) \Big) + o_s(1)
		\end{equation}
		as $y\to \infty$, where $\sigma := \Re s$. In fact, we show that 
		\begin{align}
			-\int_{1/y}^y \frac{t^{s-1}}{e^{2\pi(t+iz)} - 1} dt &= \int_{1/y}^y t^{s-1}\sum_{n\ge 0}e^{2\pi n(t+iz)}dt = \frac{(y^s - y^{-s})}{s} + \sum_{j= 0}^{\lfloor \sigma \rfloor -1} \frac{(-1)^j(s-1)_j y^{s-1-j}}{(2\pi)^{j+1}} \mr{Li}_{j+1}(e(x)) + o_s(1) \label{eqn::exponential_int_1/y_to_y}\\
			\int_y^\infty \frac{t^{s-1}}{e^{2\pi(t+iz)} - 1} dt &= \int_y^\infty t^{s-1} \sum_{n\ge 1} e^{-2\pi n(t+iz)}dt =  \sum_{j= 0}^{\lfloor \sigma \rfloor -1} \frac{(s-1)_j y^{s-1-j}}{(2\pi)^{j+1}} \mr{Li}_{j+1}(e(-x)) + o_s(1) \label{eqn::exponential_int_y_to_infty}
		\end{align}
		as $y\to \infty$, where $e(x):= e^{2\pi i x}$ and the $o_s(1)$ piece is an entire function of $s$. By continuity, we may assume without loss of generality that $s\notin \Z$. Notice that (\ref{eqn::exponential_int_1/y_to_y}) and (\ref{eqn::exponential_int_y_to_infty}) imply (\ref{eqn::exponential_int_1/y_to_infty}) and thus the theorem. 
		
		We first prove equation (\ref{eqn::exponential_int_y_to_infty}). Note first that 
		\begin{align*}
			\sum_{n\ge 1} \left|((y+t)^{s-1} - y^{s-1}) e^{-2\pi n t}e(-nx) \right| &\le |(y+t)^{s-1} - y^{s-1}| \sum_{n\ge 1}e^{-2\pi n t} = \frac{|(t+y)^{s-1} - y^{s-1}|}{e^{2\pi t} - 1}
		\end{align*}
		and since \[\lim_{t\to 0^+} \frac{|(t+y)^{s-1} - y^{s-1}|}{e^{2\pi t} - 1} = \lim_{t\to 0^+}\frac{|(t+y)^{s-1} - y^{s-1}|}{t} \cdot\frac{t}{e^{2\pi t} - 1} = \frac{|s-1|y^{\sigma -2}}{2\pi},\]
		it follows that $ \frac{|(t+y)^{s-1} - y^{s-1}|}{e^{2\pi t} - 1}$ is integrable from $0$ to $\infty$. By dominated convergence we have 
		\begin{align*}
			\int_0^\infty \frac{(y+t)^{s-1} - y^{s-1}}{e^{2\pi t}e(x) - 1}dt &= \sum_{n\ge 1} e(-nx) \int_0^\infty \left( (y+t)^{s-1} - y^{s-1} \right) e^{-2\pi n t}dt \\
			&= \sum_{n\ge 1} e(-nx) \left( \frac{e^{2\pi n y}\Gamma(s, 2\pi n y)}{(2\pi n )^s} - \frac{y^{s-1}}{2\pi n } \right) \\
			&= \sum_{n\ge 1} \frac{e(-nx)y^{s-1}}{2\pi n} \left( \frac{\Gamma(s, 2\pi n y)}{(2\pi n y)^{s-1}e^{-2\pi n y}} - 1\right)
		\end{align*}
		Let $k = \max\{\lfloor \sigma \rfloor, 1\}$. By 6.5.32 of \cite{handbook_fns}, we have \[\left| \frac{\Gamma(s, 2\pi n y)}{(2\pi n y)^{s-1}e^{-2\pi n y}} - \sum_{j=0}^{k-1} \frac{(s-1)_j}{(2\pi n y)^j} \right| \ll \frac{|(s-1)_{k}|}{(2\pi n y)^{k}}\]
		and so 
		\[\sum_{n\ge 1} \left|\frac{e(-nx)y^{s-1}}{2\pi n} \left( \frac{\Gamma(s, 2\pi n y)}{(2\pi n y)^{s-1}e^{-2\pi n y}} - \sum_{j=0}^{k-1} \frac{(s-1)_j}{(2\pi n y)^j}\right) \right| \ll \sum_{n\ge 1} \frac{y^{\sigma - 1}}{n} \cdot \frac{|(s-1)_{k}|}{(2\pi n y)^{k}} = \frac{|(s-1)_{k}| \zeta(k+1)}{(2\pi)^{k}} \cdot y^{\sigma - k-1}\]
		where the bound is uniform in both $y$ and $s$. This implies that \begin{equation}\label{eqn::controlled_sum_of_incomplete_gamma_factors}
			\sum_{n\ge 1} \frac{e(-nx)y^{s-1}}{2\pi n} \left( \frac{\Gamma(s, 2\pi n y)}{(2\pi n y)^{s-1}e^{-2\pi n y}} - \sum_{j=0}^{k-1} \frac{(s-1)_j}{(2\pi n y)^j}\right)  = o(1)
		\end{equation}
		as $y\to \infty$, regardless of $s\in \C$. Now, notice that for $j\ge 0$, we have \[(s-1)_j \sum_{n\ge 1} \frac{e(-nx) y^{s-1-j}}{(2\pi n)^{j+1}} = \frac{(s-1)_j y^{s-1-j}}{(2\pi)^{j+1}} \mr{Li}_{j+1}(e(-x)) = \int_0^\infty \frac{(s-1)_j t^j y^{s-1-j}}{e^{2\pi t}e(x) - 1} dt.\]
		Hence, 
		\begin{align*}
			 \int_y^\infty \frac{t^{s-1}}{e^{2\pi(t+iz)} - 1} dt &= \int_0^\infty \frac{(y+t)^{s-1} - y^{s-1} }{e^{2\pi t}e(x) - 1} dt - \sum_{j=1}^{k-1} \int_0^\infty \frac{(s-1)_j t^j y^{s-1-j}}{e^{2\pi t}e(x) - 1}dt +\sum_{j= 0}^{k-1} \frac{(s-1)_j y^{s-1-j}}{(2\pi)^{j+1}} \mr{Li}_{j+1}(e(-x))   \\
			&= \sum_{j= 0}^{k-1} \frac{(s-1)_j y^{s-1-j}}{(2\pi)^{j+1}} \mr{Li}_{j+1}(e(-x)) + \sum_{n\ge 1} \frac{e(-nx)y^{s-1}}{2\pi n} \left( \frac{\Gamma(s, 2\pi n y)}{(2\pi n y)^{s-1}e^{-2\pi n y}} - \sum_{j=0}^{k-1} \frac{(s-1)_j}{(2\pi n y)^j}\right).
		\end{align*}
		By (\ref{eqn::controlled_sum_of_incomplete_gamma_factors}), the series on the right is $o(1)$ as $y\to \infty$. If $\sigma = 1$, then $k = \lfloor \sigma \rfloor$ and we have established (\ref{eqn::exponential_int_y_to_infty}). On the other hand, if $\sigma < 1$ then $k=1$, and so the required integral is just $\frac{1}{2\pi} y^{s-1} + o(1)$. However, for $\sigma<1$, we have $y^{s-1} = o(1)$ as well, and so the above calculation implies (\ref{eqn::exponential_int_y_to_infty}) in this case as well.
		
		We now prove (\ref{eqn::exponential_int_1/y_to_y}). Similar to above, it is easy to see that $ \frac{|(y-t)^{s-1} - y^{s-1}|}{e^{2\pi t} - 1}$ is integrable from $0$ to $y-1/y$. By dominated convergence,
		\begin{align*}
			\int_{1/y}^y (t^{s-1} - y^{s-1})\sum_{n\ge 1} e^{2\pi i n z} e^{2\pi n t} dt &= \int_0^{y-1/y} ((y-t)^{s-1} - y^{s-1})\sum_{n\ge 1} e(nx) e^{-2\pi n t} dt \\
			&=\sum_{n\ge 1} e(nx) \left( \int_0^{y-1/y} (y-t)^{s-1}e^{-2\pi n t}dt - \int_0^{y-1/y} y^{s-1}  e^{-2\pi n t}dt\right) \\
			&= \sum_{n\ge 1} e(nx)\left(\frac{e^{-2\pi n y}}{(2\pi n)^s} \int_{2\pi n / y}^{2\pi n y} t^{s-1}e^t dt - \frac{y^{s-1}}{2\pi n} \left( 1 - e^{-2\pi n(y-1/y)}\right)  \right).
		\end{align*}
		
		Now, for $\sigma > 0$ and for any $c>0$ we have  
		\begin{align*}
			\int_0^c t^{s-1} e^t dt &= c^s \int_0^1 t^{s-1} e^{ct} dt = c^s \int_0^1 e^{ct} t^{s-1} (1-t)^{(s+1) - s - 1} dt \\
			&= c^s \frac{\Gamma(s) \Gamma(s+1-s)}{\Gamma(s+1)} M(s, s+1, c) = \frac{c^s}{s} M(s,s+1,c) = \frac{c^s e^c}{s} M(1, s+1, -c)
		\end{align*}
		where $M(a,b,c)$ is Kummer's function of the first kind \cite[][{}(13.2.1)]{handbook_fns}. The last equality follows from \cite[][{}(13.1.27)]{handbook_fns}. Thus \[\int_{2\pi n/y}^{2\pi n y} t^{s-1} e^tdt = \frac{(2\pi n y)^s e^{2\pi n y}}{s} M(1, s+1, -2\pi n y) - \frac{(2\pi n)^s y^{-s}e^{2\pi n / y}}{s} M(1, s+1, -2\pi n /y).\]
		Since the left hand side is an entire function in $s$, and the right side meromorphic in $s$ with possible poles at $s\in \Z_{\le 0}$, it follows that the above identity holds for all $s\in \C$ (with the poles canceling). Thus,
		\[\begin{split}
			\int_{1/y}^y (t^{s-1} - y^{s-1}) \sum_{n\ge 1} e^{2\pi i n z} e^{2\pi n t} dt = \sum_{n\ge 1} e(nx)&\Bigg( \frac{y^s}{s} M(1, s+1, -2\pi n y) -\frac{y^{-s}e^{-2\pi n y}e^{2\pi n / y}}{s} M(1, s+1, -2\pi n /y)  \\
			& \qquad- \frac{y^{s-1}}{2\pi n} + \frac{y^{s-1}}{2\pi n}  e^{-2\pi n(y-1/y)}  \Bigg)
		\end{split}\]
		
		We now assume that $s\notin \Z$. As before let $k = \max\{\lfloor \sigma \rfloor, 1\}$. We then (formally) write 
		\[\begin{split}
			\int_{1/y}^y t^{s-1} \sum_{n\ge 0} e^{2\pi i n z} e^{2\pi i n t} dt 
			={}&\int_{1/y}^y t^{s-1}dt + y^{s-1}\int_{1/y}^y \sum_{n\ge 1} e^{2\pi i n z} e^{2\pi n t} dt + \sum_{n\ge 1}\frac{e(nx)y^{s-1}}{2\pi n}\sum_{j=1}^{k-1} \frac{(-1)^j (s-1)_j}{(2\pi n y)^j} \\
			&  + \labelunder{S_1}{\sum_{n\ge 1}\frac{e(nx)y^s}{s}\left(M(1,s+1,-2\pi n y) - \frac{s}{2\pi n y} \sum_{j=0}^{k-1} \frac{(-1)^j (s-1)_j}{(2\pi n y)^j} \right)} \\
			&-  \labelunder{S_2}{\sum_{n\ge 1}\frac{e(nx)e^{-2\pi n (y - 1/y)}}{sy^s} M\left(1, s+1, -\frac{2\pi n}{y} \right)}  + \labelunder{S_3}{\sum_{n\ge 1}\frac{e(nx)e^{-2\pi n (y - 1/y)} y^{s-1}}{2\pi n}} .
		\end{split}\]
		We analyze each of the terms on the right hand side separately.
		We directly have \[S_3 = \frac{y^{s-1}}{2\pi}\mr{Li}_1\left(e^{2\pi i z + 2\pi /y} \right).\]
		Also  
		\begin{align*}
			y^{s-1}\int_{1/y}^y \sum_{n\ge 1} e^{2\pi i n z} e^{2\pi n t} dt &= y^{s-1}\int_{1/y}^y \frac{e^{2\pi (t+iz)} dt}{1-e^{2\pi (t+iz)}} = y^{s-1}\int_0^{y-1/y} \frac{dt}{e^{2\pi t}e(-x) - 1} \\
			&= y^{s-1} \left(\int_0^\infty \frac{dt}{e^{2\pi t}e(-x) - 1} - \int_{y-1/y}^\infty \frac{dt}{e^{2\pi t}e(-x) - 1} \right) \\
			&= y^{s-1} \left(\int_0^\infty \frac{dt}{e^{2\pi t}e(-x) - 1} - \int_{0}^\infty \frac{dt}{e^{2\pi t}e^{2\pi (y-1/y)}e(-x) - 1} \right) \\
			&= y^{s-1} \left(\frac{1}{2\pi} \mr{Li}_1(e(x)) - \frac1{2\pi} \mr{Li}_1\left(e^{2\pi iz + 2\pi/y)} \right) \right) = \frac{y^{s-1}}{2\pi} \mr{Li}_1(e(x)) - S_3,
		\end{align*}
		where we used the integral representation of the polylogarithm. Next, 
		\begin{align*}
			\sum_{n\ge 1}\frac{e(nx)y^{s-1}}{2\pi n}\sum_{j=1}^{k-1} \frac{(-1)^j (s-1)_j}{(2\pi n y)^j} &= \sum_{j=1}^{k-1}  \frac{(-1)^j (s-1)_j y^{s-j-1} }{(2\pi)^{j+1}}\mr{Li}_{j+1}(e(x))\\
			&= \sum_{j=0}^{k-1}  \frac{(-1)^j (s-1)_j y^{s-j-1} }{(2\pi)^{j+1}}\mr{Li}_{j+1}(e(x)) - \frac{y^{s-1}}{2\pi} \mr{Li}_1(e(x)).
		\end{align*}
		For $S_1$, notice that 13.5.1 from \cite{handbook_fns} says \[M(1, s+1, -2\pi n y) = \frac{s}{2\pi n y} \left(\sum_{j=0}^{k-1} \frac{(-1)^j (s-1)_j}{(2\pi n y)^j} + O\left( \frac{(s-1)_{k}}{(2\pi ny)^{k}} \right) \right)\]
		and so 
		\begin{align*}
			\sum_{n\ge 1}\left|\frac{e(nx)y^{s-1}}{2\pi n}\left(\frac{M(1,s+1,-2\pi n y)}{s/2\pi n y} - \sum_{j=0}^{k-1} \frac{(-1)^j (s-1)_j}{(2\pi n y)^j} \right) \right| &\ll \sum_{n\ge 1}\frac{y^{\sigma - k-1} |(s-1)_{k}| }{(2\pi n)^{k+1}} = \frac{|(s-1)_{k}| \zeta(k+1)}{(2\pi)^{k+1}} y^{\sigma - k-1},
		\end{align*}
		which implies that $S_1 = o(1)$ as $y\to \infty$. It remains to check that $S_2$ converges and $S_2 =o(1)$ as $y\to \infty$. We break the sum into two as:
		\[\sum_{n\ge 1} \frac{e(nx)}{sy^s}e^{-2\pi n (y - 1/y)} M\left(1, s+1, -\frac{2\pi n}y \right) = \left(\sum_{n<Cy} + \sum_{n>Cy} \right) \frac{e(nx)}{sy^s}e^{-2\pi n (y - 1/y)} M\left(1, s+1, -\frac{2\pi n}y \right).\]
		for some fixed constant $C>0$ chosen so that $|M(1, s+1, -2\pi w)| \ll \frac{|s|}{w}$ holds for all real $w>C$. The first sum is bounded by
		\begin{align*}
			\sum_{n<Cy} \left|\frac{e(nx)}{sy^s}e^{-2\pi n (y - 1/y)} M\left(1, s+1, -\frac{2\pi n}y \right)\right| & \le \sum_{n<Cy} \frac{1}{|s|y^\sigma}e^{-2\pi n (y - 1/y)} \sum_{h \ge 0} \frac{(2\pi n/y)^h}{|(s+1)^{(h)}|} \\
			&= \frac1{|s|} \sum_{h \ge 0} \frac{(2\pi/y)^h}{ |(s+1)^{(h)}|y^\sigma}\sum_{n<Cy} n^h e^{-2\pi n (y - 1/y)} \\
			&\le \frac1{|s|} \sum_{h \ge 0} \frac{(2\pi/y)^h}{ |(s+1)^{(h)}|y^\sigma} (Cy) \cdot (Cy)^h e^{-2\pi (y - 1/y)} \\
			&= \frac{Cy^{1-\sigma}}{|s|} e^{-2\pi (y - 1/y)} \sum_{h \ge 0} \frac{(2\pi C)^h}{ |(s+1)^{(h)}|} \\
			&\ll_s \frac{y^{1-\sigma}}{|s|} e^{-2\pi (y - 1/y) } \sum_{h \ge 0} \frac{(1)^{(h)}}{ (|\sigma|+1)^{(h)}}\frac{(2\pi C)^h}{h!} \\
			&= \frac1{|s|} y^{1-\sigma} e^{-2\pi (y-1/y)} M(1, |\sigma|+1, 2\pi C) 
		\end{align*}
		where $\xi^{(h)} := \xi(\xi+1)\cdots (\xi + h-1)$ is the rising factorial. It follows that the above bound goes to 0 as $y\to \infty$. For the second sum, we get 
		\begin{align*}
			\sum_{n>Cy} \left|\frac{e(nx)}{sy^s}e^{-2\pi n (y - 1/y)} M\left(1, s+1, -\frac{2\pi n}y \right)\right| &\ll \sum_{n>Cy} \frac{1}{|s|y^\sigma}e^{-2\pi n (y - 1/y)} \frac{|s|}{n/y} \\
			&= y^{1-\sigma}\sum_{n>Cy} \frac1n e^{-2\pi n (y-1/y)} \\
			&\ll y^{1- \sigma} \mr{Li}_1\left(e^{-2\pi(y - 1/y)} \right)
		\end{align*}
		where the latter bound clearly goes to 0 as $y\to \infty$, regardless of $s\in \C\bs \Z$. Hence (\ref{eqn::exponential_int_1/y_to_y}) follows, finishing the proof.
	\end{proof}

	\printbibliography
\end{document}